\documentclass[a4paper,10pt]{article}
\usepackage{mathtools}
\usepackage{geometry}
\usepackage[pdftex]{graphicx}
\usepackage{amsfonts,amssymb,amsmath,amsthm}
\usepackage{epstopdf}
\usepackage{csquotes}
\usepackage{color}
\usepackage{tcolorbox}
\usepackage{enumerate}
\usepackage{algorithm2e}
\usepackage[title]{appendix}
\usepackage{cite}
\usepackage{pdfpages}
\usepackage{authblk}
\usepackage{titling}
\RequirePackage[OT1]{fontenc}
\RequirePackage{amsthm,amsmath}
\RequirePackage[numbers]{natbib}
\RequirePackage[colorlinks,citecolor=blue,urlcolor=blue]{hyperref}

\usepackage{siunitx}
\usepackage{booktabs}
\usepackage{lipsum}
\usepackage{amsfonts}
\usepackage{graphicx}
\usepackage{epstopdf}
\usepackage{algorithmic}
\usepackage{bbm}
\ifpdf
  \DeclareGraphicsExtensions{.eps,.pdf,.png,.jpg}
\else
  \DeclareGraphicsExtensions{.eps}
\fi

\usepackage{bm}
\usepackage{color}
\usepackage{url}
\usepackage{hyperref}
\usepackage{lineno}
\usepackage{enumerate}
\usepackage{bbm}

\usepackage{enumitem}
\usepackage{graphicx}
\usepackage{amssymb}

\usepackage{bm}
\usepackage{amssymb}
\usepackage{amsmath}
\usepackage{amsthm}
\usepackage{mathtools}

\DeclarePairedDelimiter\floor{\lfloor}{\rfloor}

\newtheorem{definition}{Definition}
\newtheorem{lemma}{Lemma}
\newtheorem{proposition}{Proposition}
\newtheorem{theorem}{Theorem}
\newtheorem{corollary}{Corollary}
\newtheorem{remark}{Remark}

\renewcommand{\v}[1]{\boldsymbol{#1}}
\newcommand{\m}[1]{\mathbf{#1}}

\newcommand{\idef}{\stackrel{\mathrm{def}}{=}}

\usepackage{amsopn}
\usepackage{array}

\modulolinenumbers[5]

\title{\vspace{-3mm}Kernel Density Estimation with Linked Boundary Conditions}
\author{Matthew J. Colbrook
\thanks{Department of Applied Mathematics and Mathematical Physics, University of Cambridge, Wilberforce Road, Cambridge CB3 0WA, UK. Email: {m.colbrook@damtp.cam.ac.uk}}, Zdravko I. Botev
 \thanks{School of Mathematics and Statistics, The University of New South Wales, Sydney, NSW 2052, Australia.},	Karsten Kuritz
\thanks{Institute for Systems Theory and Automatic Control, University of Stuttgart, 70569 Stuttgart, Germany.},	Shev MacNamara
  \thanks{ARC Centre of Excellence for Mathematical and Statistical Frontiers, School of Mathematical and Physical Sciences, University of Technology Sydney, NSW 2007,  Australia.}
}
\date{}

\begin{document}

\maketitle

\vspace{-10mm}

\begin{abstract}
Kernel density estimation on a finite interval poses an outstanding challenge because of the well-recognized bias at the boundaries of the interval. Motivated by an application in cancer research, we consider a boundary constraint linking the values of the unknown target density function at the boundaries. We provide a kernel density estimator (KDE) that successfully incorporates this linked boundary condition, leading to a non-self-adjoint diffusion process and expansions in non-separable generalized eigenfunctions. The solution is rigorously analyzed through an integral representation given by the unified transform (or Fokas method). The new KDE possesses many desirable properties, such as consistency, asymptotically negligible bias at the boundaries, and an increased rate of approximation, as measured by the AMISE. We apply our method to the motivating example in biology and provide numerical experiments with synthetic data, including comparisons with state-of-the-art KDEs (which currently cannot handle linked boundary constraints). Results suggest that the new method is fast and accurate. Furthermore, we demonstrate how to build statistical estimators of the boundary conditions satisfied by the target function without apriori knowledge. Our analysis can also be extended to more general boundary conditions that may be encountered in applications.
\end{abstract}

\vspace{2mm}

\noindent{}\textit{Keywords:} density estimation, diffusion, unified transform, linked boundary
conditions, boundary bias, biological cell cycle.

\section{Introduction and Background}
\label{sec:Introduction}

Suppose we are given an independent and identically distributed sample $X_1,\ldots,X_n$ from some unknown density function $f_X$. Throughout, we will use a subscript $X$ in $f_X$ to indicate that $f_X$ is the probability density function of the random variable $X$. We will also denote expectation and variance with respect to $f_X$ by $\mathbb{E}_{f_X}$ and $\mathrm{Var}_{f_X}$ respectively. Estimating the density $f_X$ is one of the most common problems for discovering patterns in statistical data \cite{chevallier2017kernel,silverman2018density,simonoff2012smoothing}. When the support of $f_X$ is the whole real line, a simple and popular non-parametric method for estimating $f_X$ is the kernel density estimator (KDE)
\begin{equation}
\label{standard}
\widehat f(x;t)=\frac{1}{n\sqrt{t}}\sum_{k=1}^n \varphi\left(\frac{x-X_k}{\sqrt{t}}\right) , 
\end{equation}
with a kernel $\varphi(x)$. A common choice is a Gaussian kernel $\varphi(x)=\exp(-x^2/2)/\sqrt{2\pi}$. Here $\sqrt{t}$ is the so-called bandwidth parameter that controls the smoothness of the estimator (see, for example, \cite{wand1994kernel,tsybakov2008introduction,silverman2018density,simonoff2012smoothing} and references therein). Another viewpoint is to connect kernel density estimation to a diffusion equation, an approach pioneered by the second author in \cite{Botev2010}. Our goal in this article is to extend this analysis to linked boundary conditions. A key tool in our analysis is the unified transform (also known as the Fokas method), a novel transform for analyzing boundary value problems for linear (and integrable non-linear) partial differential equations \cite{fokas1997,fokas2002integrable,fokas2005transform,fokasbook2008,trogdon2012solution,deconinck2014non,deconinck2017fokas,colbrook2018fokas,colbrook2019unified,colbrook2019spectral,colbrook2019hybrid,colbrook2020extending,sheils2020time}. An excellent pedagogical review of this method can be found in the paper of Deconinck, Trogdon \& Vasan \cite{deconinck2014method}.

It is well-known that $\widehat f(x;t)$ is not an appropriate kernel estimator when $f_X$ has compact support \cite{geenens2014probit}, which (without loss of generality) we assume to be the unit interval $[0,1]$. The main reason for this is that $\widehat f(x;t)$ exhibits significant boundary bias at the end-points of the interval. For example, with a Gaussian kernel, no matter how small the bandwidth parameter, $\widehat f(x;t)$ will have non-zero probability mass outside the interval $[0,1]$. Various solutions have been offered to cope with this boundary bias issue, which may be classified into three main types:
\begin{enumerate}[label=(\alph*)]
\item Using special (non-Gaussian) kernels with support on $[0,1]$ or on $[0,\infty)$, as in \cite{chen1999beta,malec2014nonparametric,scaillet2004density};
\item Adding bias-correction terms to $\widehat f(x;t)$ as in \cite{dai2010simple,karunamuni2005boundary};
\item Employing domain transformations \cite{geenens2014probit,marron1994transformations}, which work by mapping the data to $(-\infty,\infty)$, constructing a KDE on the whole real line, and finally mapping the estimate back to $[0,1]$. 
\end{enumerate}

Additionally, sometimes we not only know that $f_X$ has support on $[0,1]$, but also have extra information about the values of $f_X$ at the boundaries. One example of this situation is what we will refer to as a \emph{linked boundary condition}, where we know apriori that 
\[
f_X(0)=r f_X(1)
\]
for some known given parameter $r\geq0$. Most of our analysis also carries over to complex $r$, as long as $r \neq -1$ (the PDE (\ref{eq:pde:candidate}) is degenerate irregular and the problem ill-posed when $r=-1$), but we focus on $r\geq0$ since in statistics $f_X\geq 0$. An example that motivated the current article arises in the field of biology \cite{Kuritz2017,Kuritz2020}, in particular cell cycle studies in cancer research. The cell cycle itself is one of the fundamentals of biology and knowledge about its regulation is crucial in the treatment of various diseases, most prominently cancer. Cancer is characterized by an uncontrolled cell growth and commonly treated with cytotoxic drugs. These drugs interfere with the cell cycle and in this way cause cancer cells to die. By studying the effect of chemicals on the cell cycle one can discover new drugs, identify potential resistance mechanisms or evaluate combinatorial therapy. These kind of studies have benefited from continued improvement in cell population analysis methods like fluorescence microscopy, flow cytometry, CyTOF or single-cell omics, where the abundance of up to thousands of cellular components for every individual cell in a population is measured. In such an experiment, cells in an unsynchronized cell population are spread over all stages of the cell cycle. Trajectory inference algorithms then reduce the dimensionality to a pseudotime scale by ordering cells in the population based on their similarity in the dataset \cite{Saelens2019}. Subsequently, mathematical methods based on ergodic principles infer molecular kinetics in the cell cycle from the distribution of cells in pseudotime. The value at the left boundary of this distribution must, because of cell division, be double the value at the right boundary. In other words, we have linked boundary conditions with the constant $r = 2$, but otherwise, we do not know the value of the density at the boundaries of the domain. The problem is described in more detail in Section~\ref{num_exam}, where we also demonstrate the estimator with linked boundary condition on a real dataset. In particular, for this example, respecting the linked boundary condition is crucial for generating the correct kinetics due to a certain mapping between pseudotime and real time. See also \cite{Kuritz2017,Kuritz2020}, for example, for further motivation and discussion. In other applications, even if we do not know the value of $r$, one can approximate the true value of $r$ which, together with the methods proposed in this article, leads to an increase in the rate of approximation of $f_X$ as the sample size $n$ becomes large (we do this for an example in Section \ref{synth_num_exam}, see also \S \ref{asymp_properties} for some results in this direction).

Unfortunately, to the best of our knowledge, all of the currently existing kernel density estimation methods, bias-correcting or not, cannot satisfactorily handle the linked boundary condition. Figure~\ref{fig:ksdensity:BAD} shows a typical example of what can go wrong when a standard density estimator is applied to real biological data.
The result is a smooth density with \textit{two unacceptable features}:
\begin{itemize}
\item The domain $x \in [0,1]$ is not respected, and instead the solution has positive density for negative values of $x$, and also for $x>1$, which are physically unreasonable. This problem can be addressed using existing bias-correction methods and is not the challenge that we had to overcome in this article.
\item The density does not respect the important biological constraint of the linked boundary condition (that the left value should be double the right, in this particular application), and instead the density decays to zero as $|x|$ becomes large. Existing bias-correction methods do not address this problem.  
\end{itemize}

\begin{figure}
\centering
    \includegraphics[scale=0.5]{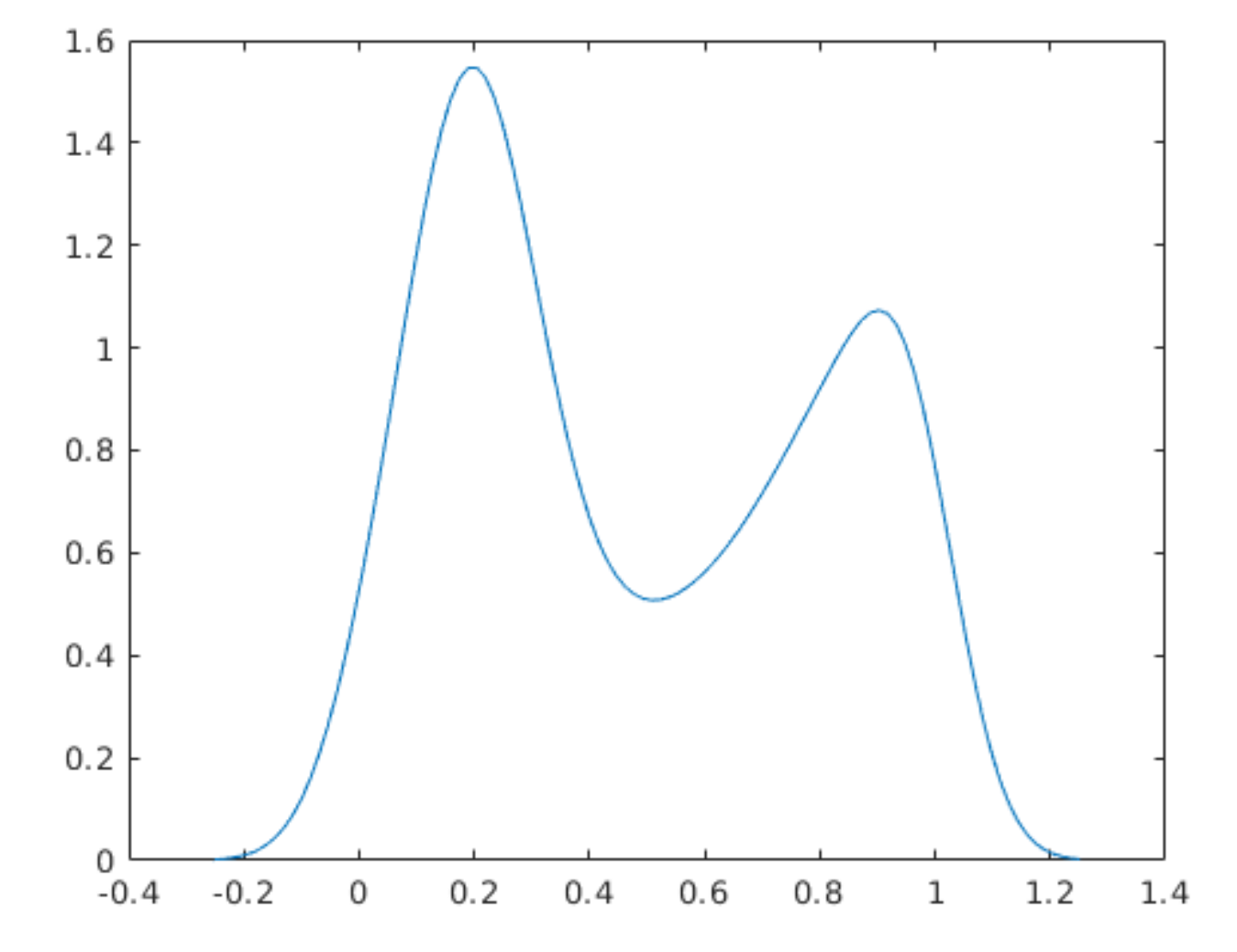}
    \caption{A typical example of output from a KDE (\texttt{ksdensity} from MATLAB) applied to our real biological data. This does not respect the domain, and it also does not respect the important linked boundary conditions. The methods that we propose in this article address those issues simultaneously, with results for this data set shown in Figure \ref{fig:Example}.}
    \label{fig:ksdensity:BAD}
\end{figure}

The purpose of this article is to describe a new KDE that can handle the more general problem of linked boundary conditions with an arbitrary value of $r$; the situation of interest in the biological application where $r=2$ is then solved as an important special case. Figure~\ref{fig:Example} (C) shows a successful application of our proposed method. The MAPiT toolbox for single-cell data analysis \cite{Kuritz2020} applies our new KDE with linked boundary conditions to analyze cell cycle dependent molecular kinetics.

Our proposed estimator is of type (a), that is, we construct a special kernel with support on $[0,1]$, and such that the linked boundary condition is incorporated into the resulting estimator. 
Our kernel is inspired by the solution of a diffusion-type PDE \cite{agarwal2010data,Botev2010,mehmood2016clustering,santhosh2013bivariate,xu2015estimating}. In particular, we modify the diffusion model in \cite{Botev2010} so that it satisfies the linked boundary conditions. Unlike the case in \cite{Botev2010}, the non-self-adjoint initial-boundary problem that arises cannot be diagonalized, meaning the solution cannot be expressed as a series solution of eigenfunctions of the spatial differential operator in the usual sense. Instead, we use the unified transform, which provides an algorithmic recipe for solving these types of problems via an integral solution. This was the way we first found the solution formula to our diffusion model, and the integral representation simplifies many of the proofs in our analysis. So far, the only case of our problem considered in the literature on this method has been $r=1$ \cite{trogdon2012solution} (periodic). For the heat equation with oblique Robin boundary conditions/non-local boundary conditions we refer the reader to \cite{mantzavinos2013unified,miller2018diffusion,pelloni2018nonlocal} and for interface problems we refer the reader to \cite{sheils2014heat,deconinck2015interface,sheils2015interface}. Recently linked boundary conditions have been considered for the Schr\"odinger equation in \cite{olver2019revivals} (however, in \cite{olver2019revivals}, the parameters were chosen such that the characteristic values were simple, in other words the eigenvalues were simple, making the analysis easier and leading to a series solution in terms of bona fide eigenfunctions).

We then construct a series expansion in non-separable generalized eigenfunctions of the spatial derivative operator by deforming the contours in the integral representation and applying Cauchy's residue theorem. This formal solution is then rigorously verified and studied via a non-symmetric heat kernel. Each of these representations (integral and series) is beneficial for different analysis. For instance, the integral representation is much easier to construct and makes it easier to study regularity properties, as well as some parts of the behavior as $t\downarrow0$, whereas the kernel representation is useful for proving conservation of mass (the solution generates a true probability measure) and studying the asymptotic mean integrated squared error (AMISE). Although it is not the goal of the present article, we envisage that the method that we demonstrate here can also be generalized to the multivariate case and to scenarios where other types of boundary conditions (such as linked derivatives or on-local boundary conditions) arise or can be estimated. In these situations, we recommend using the unified transform to find the solution of the resulting PDE. For numerical implementation of the unified transform, we refer the reader to \cite{de2019hybrid}.

We also consider the discrete counterpart of the continuous model for two reasons. First, it is a numerical approximation to the continuous model and a useful way to compute the solution of the PDE. Second, the discrete model is relevant when we deal with data which is already pre-binned.

The rest of the article is organized as follows. In the next section, we describe the continuous model for the application at hand. Our results provide the necessary assurances that the PDE model is a valid and accurate density estimator. We then discuss the issue of choosing an optimal bandwidth (stopping time for the PDE model), including pointwise bias, asymptotic properties and the AMISE. We briefly discuss numerical methods for calculating the estimator and, in particular, a discretized version of the continuous PDE, which we prove converges to the unique continuous solution. Finally, the new method is applied to a real dataset from a biological application in Section~\ref{num_exam}, and we also provide an illustrative set of examples with synthetic datasets. We compare our new estimator to several well-known methods and these results suggest that our new method is typically more accurate and faster, and that it does not suffer from boundary bias. We finish with a short conclusion.

All technical analysis and proofs are moved to the appendices to ensure that the presentation flows more easily. Freely available code for the new kernel methods is also provided at {\url{https://github.com/MColbrook/Kernel-Density-Estimation-with-Linked-BCs}}.

\section{The Continuous Linked--Boundaries Model}
\label{cts_model_properties}

In this section, we present the continuous diffusion model that satisfies the linked boundary condition and discuss the analytical properties of its solution. Our proposed diffusion model for a linked-boundary KDE is the solution of the formal PDE system:
\begin{equation}
\begin{split}
\frac{\partial f}{\partial t} &=\frac{1}{2}\frac{\partial ^2 f}{\partial x^2} ,\qquad x\in [0,1], \;\;\;t>0,\\
\mathrm{IC:}\quad \lim_{t\downarrow 0}f(\cdot,t)&=f_0,\\
\mathrm{BCs:}\quad f(0,t)&=    r  f(1,t), \quad \frac{\partial f }{\partial x} (0,t)= \frac{\partial f}{\partial x} (1,t), \quad\forall t>0.
\end{split}
 \label{eq:pde:candidate}
 \end{equation}
The boundary condition $\frac{\partial f }{\partial x} (0,t)= \frac{\partial f}{\partial x} (1,t)$ is enforced so that the solution at any time $t\geq 0$ gives a probability measure (see Theorem \ref{cons_thm}). When considering the setup described in the introduction, the initial condition is given by 
\begin{equation}
\label{IC} \textstyle
f_0=\frac{1}{n}\sum_{k=1}^n\delta_{X_k},
\end{equation}
the empirical measure of the given sample $X_1,\ldots,X_n$. In other words, $f_0$ is a sum of Dirac delta distributions. However, in our analysis we also consider more general initial conditions. Many of the existence and uniqueness theorems carry over from the well-known $r=1$ (periodic) case. In particular, the boundary conditions and PDE make sense when the initial data is given by a finite Borel measure, which we also denote by $f_0$. Sometimes we will also refer to a function $g$ as a measure through the formula $g(U)=\int_Ug(x)dx$ for Borel sets $U$. Therefore, since the initial data is a distribution, we need to be precise by what we mean when writing $\lim_{t\downarrow 0}f(\cdot,t)=f_0$.

\begin{definition}
Denote the class of finite Borel measures on $[0,1]$ by $M([0,1])$ and equip this space with the vague topology (i.e. weak$^*$ topology). We let $C^w(0,T;M([0,1]))$ denote the space of all continuous maps
\begin{align*}
&\mu:[0,T)\rightarrow M([0,1]),\\
&\mu(t)=\mu_{t}.
\end{align*}
In other words, $\mu_{t}$ is continuous as a function of $t$ in the vague topology, meaning that for any given function $g$ that is continuous on the interval $[0,1]$, the integral $\int_0^1g(x)d\mu_t(x)$ is continuous as a function of $t\in[0,T)$.
\end{definition}

We look for weak solutions of (\ref{eq:pde:candidate}). In terms of notation, we will denote the derivative with respect to $x$ by $g_x$ and use $\mu(g)$ to denote the integration of a function $g$ against a measure $\mu$. The following adjoint boundary conditions are exactly those that arise from formal integration by parts.

\begin{definition}
Let $\mathcal{F}(r)$ denote all $g\in C^{\infty}([0,1])$ satisfying the adjoint linked boundary conditions
\begin{equation}
\label{adjoint_bcs}
g(1)=g(0),\quad g_x(1)=r g_x(0).
\end{equation}
\end{definition}

\begin{definition}[Weak Solution]
Let $f_0\in M([0,1])$ such that $f_0(\{0\})=r f_0(\{1\})$. We say that $\mu\in C^w(0,T;M([0,1]))$ is a weak solution to (\ref{eq:pde:candidate}) for $t\in [0,T)$ if $\mu_0=f_0$ and for all $g\in \mathcal{F}(r)$, $\mu_t(g)$ is differentiable for $t>0$ with
\begin{equation}
\label{weak_sol}
\frac{d}{dt}\mu_t(g)=\frac{1}{2}\mu_t(g_{xx}).
\end{equation}
\end{definition}

We can now precisely state the well-posedness of (\ref{eq:pde:candidate}).

\begin{theorem}[Well-Posedness]
\label{wk_thm}
Assume our initial condition $f_0$ lies in $M([0,1])$ and satisfies $f_0(\{0\})=r f_0(\{1\})$. Then there exists a unique weak solution to (\ref{eq:pde:candidate}) for $t\in [0,T)$ for any $T\in(0,\infty]$, which we denote by $f(\cdot,t)$. For $t>0$ this weak solution is a function that is smooth in $t$ and real analytic as a function of $x$. Furthermore, the solution has the following properties which generalize the classical periodic case of $r=1$:

\begin{enumerate}
	\item If $f_0\in C([0,1])$ (the space of continuous functions on $[0,1]$), then for any $x\in (0,1)$, $f(x,t)$ converges to $f_0(x)$ as $t\downarrow0$. If $f_0(0)=r f_0(1)$ then $f(\cdot,t)$ converges to $f_0$ as $t\downarrow0$ uniformly over the whole closed interval $[0,1]$.
	\item If $1\leq p<\infty$ and $f_0\in L^p([0,1])$, then $f$ is the unique weak solution in $C(0,T;L^p([0,1]))$ and $f(\cdot,t)$ converges to $f_0$ as $t\downarrow0$ in $L^p([0,1])$.
\end{enumerate}
\end{theorem}

\begin{proof}
See Appendix \ref{well-posedness}.
\end{proof}

The system \eqref{eq:pde:candidate} is a natural candidate for density estimation with such a linked boundary condition. Whilst Theorem \ref{wk_thm} is expected and analogous to the $r=1$ case, due to the non-self-adjoint boundary conditions, it is not immediately obvious what properties solutions of (\ref{eq:pde:candidate}) have. For example, one question is whether or not the solution is a probability density for $t>0$, and what its asymptotic properties are. Moreover, we would like to be able to write down an explicit solution formula (and ultimately use this to numerically compute the solution), a formal derivation of which is given in Appendix \ref{sol_deriv} using the unified transform.

\subsection{Solution formula and consistency of KDE at boundaries}

If we ignore the constant $r$ in the boundary conditions of \eqref{eq:pde:candidate} (and replace it by the special case $r=1$), then we would have the \textit{simple diffusion equation with periodic boundary conditions.} One can successfully apply Fourier methods, separation-of-variables or Sturm--Liouville theory to solve the periodic version of this PDE \cite{EvansPDEBook,gilbarg2015elliptic}.
However, when $r \ne 1$, making the ansatz that a solution is of the `rank one', separable form $f(x,t)=g(x)h(t)$ leads to a non-complete set of functions and \textit{separation of variables fails.} The differential operator associated with the evolution equation in (\ref{eq:pde:candidate}) is regular in the sense of Birkhoff \cite{birkhoff1908boundary} but not self-adjoint when $r\neq1$, due to the boundary conditions. Nevertheless, it is possible to generalize the notion of eigenfunctions of the differential operator \cite{coddington1955theory} and these generalized eigenfunctions form a complete system in $L^2([0,1])$ \cite{naimark1952linear,locker2000spectral} (and in fact form a Riesz basis). This allows us to obtain a series expansion of the solution. The easiest way to derive this is through the unified transform, which also generates a useful integral representation.

\begin{theorem}[Integral and Series Representations of Diffusion Estimator]
\label{rep_thm}
Suppose that the conditions of Theorem \ref{wk_thm} hold. Then the the unique solution of \eqref{eq:pde:candidate} has the following representations for $t>0$.

\vspace{1mm}

\noindent{}\textbf{Integral representation:}
	\begin{equation}
\label{thm_st}
\begin{split}
&2\pi f(x,t)=\int_{-\infty}^\infty {\exp(ikx-k^2t/2)}\hat{f}_0(k)dk\\
&-  \textstyle  \int_{\partial D^+}\frac{\exp(ikx-k^2t/2)}{{\Upsilon(k)}}\left\{\hat{f}_0(k)[(1+r)\exp(ik)-2r]+\hat{f}_0(-k)(1-r)\exp(-ik)\right\}dk\\
&- \textstyle\int_{\partial D^-}\frac{\exp(ik(x-1)-k^2t/2)}{\Upsilon(k)}\left\{\hat{f}_0(k)[2\exp(ik)-(1+r)]+\hat{f}_0(-k)(1-r)\right\}dk.
\end{split}
\end{equation}
Here the contours $\partial D^{\pm}$ are shown in Figure \ref{domains} and are deformations of the boundaries of $D^{\pm}=\{k\in\mathbb{C}^{\pm}:\mathrm{Re}(k^2)<0\}$. The determinant function is given by $\Upsilon(k)=2(1+r)(\cos(k)-1)$ and
$
\hat{f}_0(k):=\int_0^1 \exp(-ikx)f_0(x)dx.
$

\vspace{1mm}

\noindent{}\textbf{Series representation:}
	\begin{equation}
\label{series2}
\begin{split}
f(x,t)=&\frac{2}{(1+r)}\hat{c}_0(0)\phi_0(x)\\
&+\sum_{n\in\mathbb{N}}\frac{4\exp(-k_n^2t/2)}{(1+r)}\big\{\hat{c}_0(k_n)\phi_n(x)-k_nt(1-r)\hat{c}_0(k_n)\sin(k_nx)\\
&\quad\quad\quad\quad\quad+[\hat{s}_0(k_n)-(1-r)\hat{s}_1(k_n)]\sin(k_nx)\big\},
\end{split}
\end{equation}
where $k_n=2n\pi$ and
\begin{align*}
&\phi_n(x)=\left(r+(1-r)x\right)\cos(k_nx),&&\hat{s}_0(k)=\int_0^1 \sin(kx) f_0(x)dx,\\
&\hat{c}_0(k)=\int_0^1 \cos(kx) f_0(x)dx,&&\hat{s}_1(k)=\int_0^1 \sin(kx) xf_0(x)dx.
\end{align*}

\end{theorem}

\begin{proof}
See Appendix \ref{well-posedness}.
\end{proof}

In the case where $r\neq 1$, in addition to the usual \textit{separable} solutions $\exp(ik_nx-k_n^2t/2)$, the series expansion also includes the \textit{non-separable} solutions $\exp(ik_nx-k_n^2t/2)(x+ik_nt)$. We can understand these as being generalized eigenfunctions in the following sense (see the early papers \cite{machover1965generalized,tamarkin1928some}). 
Define the operator
\begin{equation}
\label{domain}
\mathbb{A}=-\frac{d^2}{dx^2},\quad \mathcal{D}(\mathbb{A})=\{u\in H^2([0,1]):u(0)=r u(1),u_x(0)=u_x(1)\},
\end{equation}
where $\mathcal{D}(\mathbb{A})$ denotes the domain of $\mathbb{A}$. We use $\mathcal{N}$ to denote the null space, which is sometimes often termed the kernel, of an operator, i.e. $\mathcal{N}(S)$ is the space of all vectors $v$ with $S(v)=0$. It is then easily checked that $\phi_n\in\mathcal{N}((\mathbb{A}-k_n^2I)^2)$. In particular, both $\phi_n$ and $(\mathbb{A}-k_n^2I)\phi_n$ satisfy the required boundary conditions. These functions \textit{block diagonalize} the operator in an analogous form to the Jordan normal form for finite matrices. If we consider any generalized eigenspace $\mathcal{N}((\mathbb{A}-k_n^2I)^2)$ corresponding to $k_n^2=4\pi^2n^2$ with $n>0$ and choose the basis $\{\sin(k_nx),\phi_n(x)/(2k_n)\}$, the operator acts on this subspace as the matrix
$$
\left(
\begin{tabular}{cc}
$k_n^2$   & $1-r$ \\
$0$       & $k_n^2$
\end{tabular}
\right),
$$
which cannot be diagonalized for $r\neq 1$.

For our purposes of kernel density estimation, we define an \textit{integral kernel} $K$ so that we can write the solution as
$$
f(x,t)=\int_0^1K(r;x,y,t)f_0(y)dy.
$$
After some residue calculus (see \eqref{solution2} in the Appendix), this is given by the somewhat complicated expression:
\begin{equation}
\label{kernel0}
\begin{split}
K(r;x,y,t)&=\sum_{n\in\mathbb{Z}}{\exp(ik_nx-k_n^2t/2)}\Big[\exp(-ik_ny)
+\frac{1-r}{1+r}(x+ik_nt)\exp(-ik_ny)\\
&+\frac{1-r}{1+r}(x+ik_nt-1)\exp(ik_ny)+\frac{1-r}{1+r}y(\exp(ik_ny)-\exp(-ik_ny))\Big],
\end{split}
\end{equation}
which can be re-expressed in terms of the more common $r=1$ kernel and its derivative, as in \eqref{kernel1}. For the initial data (\ref{IC}) this gives the density estimate
$$
f(x,t)=\frac{1}{n}\sum_{k=1}^nK(r;x,X_k,t),
$$
a generalization of (\ref{standard}). A key consequence of the solution from Theorem \ref{rep_thm} is that the pointwise bias of the corresponding diffusion estimator vanishes if $f_X$ is continuous with $f_X(0)=rf_X(1)$. Namely, we have the following.

\begin{theorem}[Consistency of Estimator at Boundaries]
\label{thm:boundary consistency}
Suppose that the initial data is given by (\ref{IC}) and that $f_X\in C([0,1])$ with $f_X(0)=rf_X(1)$. Then the solution of the PDE \eqref{eq:pde:candidate} satisfies
\begin{equation}
\lim_{t\downarrow0} \mathbb{E}_{f_X}(f(x,t))=f_X(x),
\end{equation}
uniformly in $x$. Further, if in addition $f_X\in C^1([0,1])$ and $x_t=x+\mathcal{O}(\sqrt{t})$, then our estimator satisfies
$$
\left|\mathbb{E}_{f_X}(f(x_t,t))-f_X(x)\right|\leq C(f_X)\sqrt{t},
$$
with $C(f_X)$ a constant independent of $x\in[0,1]$, but dependent on the true $f_X$.
\end{theorem}

\begin{proof}
See Appendix \ref{consist_append}.
\end{proof}

\begin{remark}
For consistency in $L^p$ spaces, we refer the reader to Proposition \ref{cty_bdy2} in Appendix \ref{consist_append}.
\end{remark}

\subsection{Conservation of probability and non-negativity}

In addition to establishing that the behavior of the PDE solution near the boundaries is satisfactory, we also want the PDE solution to be a proper bona fide probability density --- a non-negative function integrating to unity. The main tool in the proof of this is the Maximum Principle \cite{EvansPDEBook,gilbarg2015elliptic} for parabolic PDEs. The Maximum Principle states that a solution of the diffusion equation attains a maximum on the `boundary' of the two-dimensional region in space $x \in [0,1]$ and time $t\ge 0$. If our initial condition is given by a continuous function, then the maximum principle gives the following.

\begin{proposition}[Bounds on Diffusion Estimator]
\label{max_prin}
Suppose that the conditions of Theorem \ref{wk_thm} hold and that $f_0$ is a continuous function with $f_0(0)=rf_0(1)$ and non-negative with $0\leq a \leq f_0(x)\leq b$ for all $x\in [0,1]$. Then for any $t>0$ and $x\in[0,1]$ we have
\begin{equation}\textstyle
\label{sol_bound}
\min\left\{\frac{2r}{1+r},\frac{2}{1+r}\right\}a\leq f(x,t)\leq \max\left\{\frac{2r}{1+r},\frac{2}{1+r}\right\}b.
\end{equation}
In particular, $f$ remains bounded away from $0$ if $a>0$ and $r>0$.
\end{proposition}

\begin{proof}
See Appendix \ref{ndprf}.
\end{proof}

However, we also want this to hold when $f_0$ is given by (\ref{IC}). Furthermore, if we start with a probability measure as our initial condition, then we want the solution to be the density function of a probability distribution for any $t>0$. In the context of density estimation, this essential property corresponds to \textit{conservation of probability.} This is made precise in the following theorem, which does not require continuous initial data.

\begin{theorem}[A Bona Fide Kernel Density Estimator]
\label{cons_thm}
Suppose that the conditions of Theorem \ref{wk_thm} hold and that the initial condition $f_0$ is a probability measure. Then,
\begin{enumerate}
	\item$\int_0^1 f(x,t)dx=1$, for $t>0$,
	\item $f(x,t)\geq 0$ for $t>0$ and $x\in [0,1]$.
\end{enumerate}
\end{theorem}
\begin{proof}
See Appendix \ref{ndprf2}.
\end{proof}

From the solution formula \eqref{series2}, we can also characterize the behavior of the solution for large bandwidths (large $t$), that is, when the estimator oversmooths the data. An example of this is given in Figure \ref{fig:toy:example}.

\begin{corollary}[Oversmoothing Behavior with Large Bandwidth]
\label{large_time}
Suppose that the conditions of Theorem \ref{wk_thm} hold, then as $t\rightarrow\infty$, $f$ converges uniformly on $[0,1]$ to the linear function
\begin{equation}
\label{inf_t} \textstyle
f_{\infty}(x):=\frac{2}{(1+r)}\hat{c}_0(0)\phi_0(x).
\end{equation}
This linear function is the unique stationary function that obeys the boundary conditions and has the same integral over $[0,1]$ as $f_0$.
\end{corollary}

\begin{figure}
\centering
    \includegraphics[width=0.5\textwidth,trim={32mm 93mm 35mm 95mm},clip]{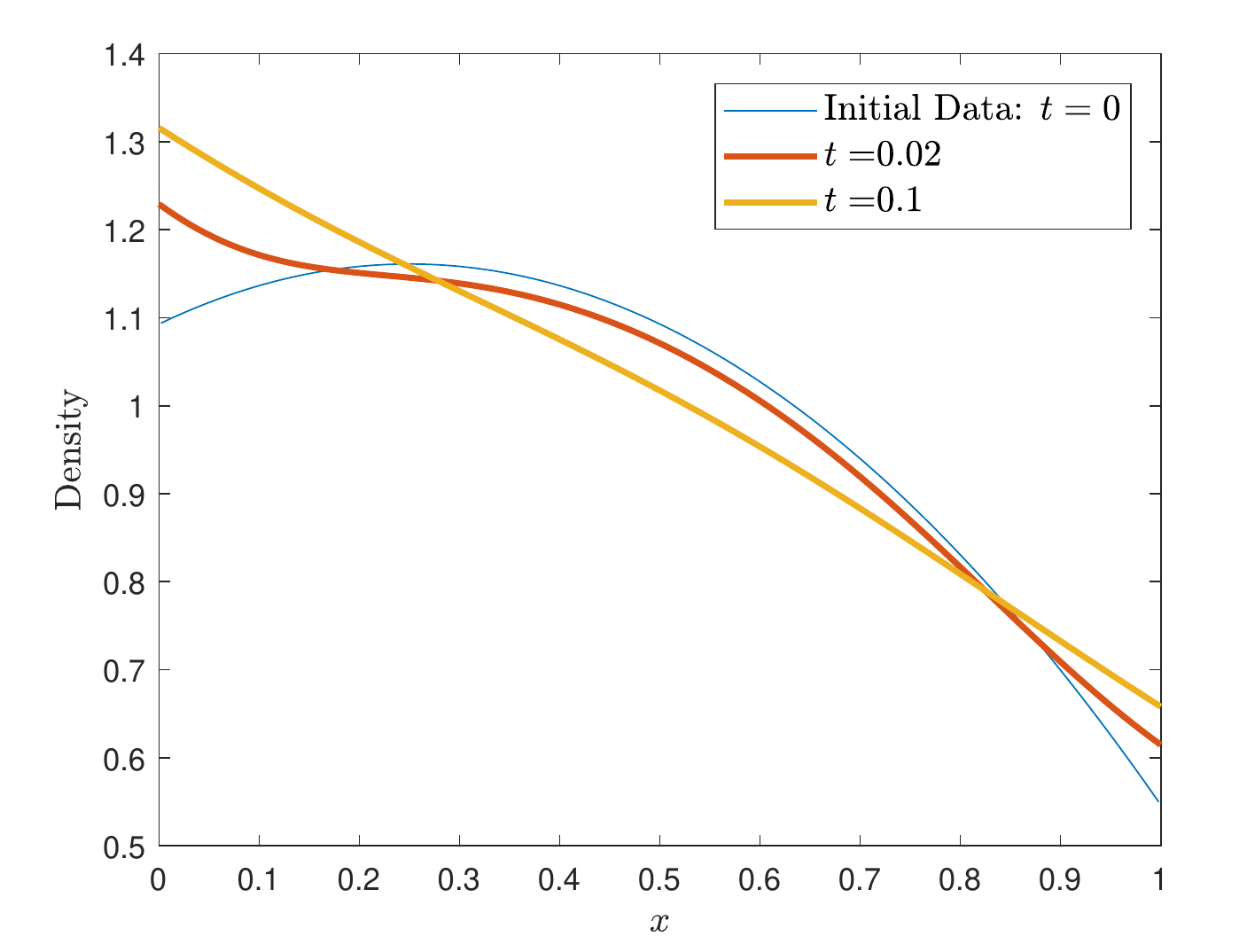}
    \caption{An example of the solution of the continuous PDE \eqref{eq:pde:candidate} at three time points, with $f_0(x) = \frac{6}{11}(-2x^2 +x+2)$. The values at the boundaries change with time, but the {ratio} remains a constant with $f(0,t)=2f(1,t)$. As $t\rightarrow\infty$, the solution converges to a straight line.}
    \label{fig:toy:example}
\end{figure}

\section{Asymptotic Properties and Bandwidth Choice}
\label{asymp_properties}

An important issue in kernel density estimation is how to choose the bandwidth parameter or, equivalently, the final or stopping time $T$ at which we compute the solution of the PDE. This issue has already received extensive attention in the literature \cite{jones1992progress,sheather1991reliable,devroye1997universal,jones1996brief}. We now give a brief summary of that issue, and we also make two suggestions for known methods already available. After that, we address the issue specifically in the context of our linked boundaries model.

At one extreme, if we choose $T=0$, then we recover the initial condition, which is precisely the raw data, with an estimator with zero bias, but infinite variance. At the other extreme, if we let $T \rightarrow \infty$, then we obtain a stationary density that is a straight line (see Corollary \ref{large_time}), which contains no information whatsoever about the raw data (other than the empirical mean), giving an estimator of zero variance, but significant bias. In between, $0<T<\infty$, we have some smoothing effect while also retaining some information from the original data --- an optimal balance between the variance and the bias of the estimator.

One would also like a consistent estimator --- as more and more data are included, it must converge to the true density (for instance, in the mean squared sense). Various proposals for the stopping times and their properties are available. One of the most common choices is `Silverman's rule of thumb' \cite{silverman2018density}, which works very well when the data is close to being normally distributed. We expect that this choice is fine for the simpler datasets and examples that we consider in this article. Another possible approach is to use the output from the freely available software of one of the authors: \url{https://au.mathworks.com/matlabcentral/fileexchange/14034-kernel-density-estimator}. This is expected to be a better choice than Silverman's rule in situations where there are many widely separated peaks in the data. In particular, \cite{Botev2010} introduced a non-parametric selection method that avoids the so-called \emph{normal reference rules} that may adversely affect plug-in estimators of the bandwidth.

We now give a more precise treatment of the choice of smoothing bandwidth for the linked boundaries model, as well as discussing the Mean Integrated Squared Error (MISE) defined by
\begin{align}
\label{MISE_def}
\mathrm{MISE}\{f\}(t)&=\mathbb{E}_{f_X}\left\{\int_0^1[f(x,t)-f_X(x,t)]^2dx\right\}\\
&=\int_0^1\mathbb\{\mathbb{E}_{f_X}[f(x,t)]-f_X(x)\}^2dx+\int_0^1\mathrm{Var}_{f_X}[f(x,t)]dx.
\end{align}
Often one is interested in the asymptotic approximation to the MISE, denoted AMISE, under the requirements that $t=t_n\downarrow0$ and $n\sqrt{t_n}\rightarrow\infty$, which ensure consistency of the estimator. The asymptotically optimal bandwidth is then the minimizer of the AMISE. For our continuous model of kernel density estimation we have the following result (proven in Appendix~\ref{AMISE_proof}) which gives the same $\mathcal{O}(n^{-4/5})$ rate of convergence as the Gaussian KDE on the whole real line.

\begin{theorem}[Asymptotic Bias and Variance of Diffusion Estimator]
\label{AMISE_1}
Let $t_n$ be such that $\lim_{n\rightarrow\infty}t_n=0$ and $\lim_{n\rightarrow\infty}n\sqrt{t_n}=\infty$ and suppose that $f_X\in C^2([0,1])$ (twice continuously differentiable) with $f_X(0)=rf_X(1)$. Then the following hold as $n\rightarrow\infty$:
\begin{enumerate}
	\item The integrated variance has the asymptotic behavior
	\begin{equation}
	\label{Var_asym}
	\int_0^1\mathrm{Var}_{f_X}[f(x,t_n)]dx\sim\frac{1}{2n\sqrt{\pi t_n}}.
	\end{equation}
	\item If $f_X'(0)=f_X'(1)$ then the integrated squared bias is
	\begin{equation}
	\label{Exp_asym}
	\int_0^1\left\{\mathbb{E}_{f_X}[f(x,t_n)]-f_X(x)\right\}^2dx\sim t_n^2\int_0^1\frac{1}{4}\left[f_X^{''}(x)\right]^2dx.
	\end{equation}
\item 	If $f_X'(0)\neq f_X'(1)$ then the integrated squared bias is
	\begin{equation}
	\label{Exp_asym2}
	\int_0^1\{\mathbb{E}_{f_X}[f(x,t_n)]-f_X(x)\}^2dx\sim t_n^{3/2}\frac{4-2\sqrt{2}}{3\sqrt{\pi}}\frac{r^2+1}{(1+r)^2}[f_X'(1)- f_X'(0)]^2.
	\end{equation}
\end{enumerate}
\end{theorem}

\begin{proof}
See Appendix \ref{AMISE_proof}.
\end{proof}

A direct consequence of this result is that we can select the stopping time $t$ or bandwidth to minimize the AMISE.

\begin{corollary}[Asymptotically Optimal Bandwidth Choices]
Combining the leading order bias and variance terms gives the asymptotic approximation
to the MISE:
\begin{enumerate}
	\item If $f_X'(1)=f_X'(0)$ then
	\begin{equation}
	\label{AMISE}
	\mathrm{AMISE}\{f\}(t_n)=\frac{1}{2n\sqrt{\pi t_n}}+t_n^2\int_0^1\frac{1}{4}\left[f_X^{''}(x)\right]^2dx.
	\end{equation}
	Hence, the square of the asymptotically optimal bandwidth is
	$$
	t^*=(2n\sqrt{\pi}\|f_X^{''}\|_{L^2}^{2})^{-2/5}
	$$
	with the minimum value
	$$
	\min_t\mathrm{AMISE}\{f\}(t)=\frac{5\|f_X^{''}\|_{L^2}^{2/5}}{2^{14/5}\pi^{2/5}}n^{-4/5}.
	$$
	\item If $f_X'(1)\neq f_X'(0)$ then
	\begin{equation}
	\label{AMISE2}
	\begin{split}
	\mathrm{AMISE}\{f\}(t_n)&=\frac{1}{2n\sqrt{\pi t_n}}+t_n^{3/2}\frac{4-2\sqrt{2}}{3\sqrt{\pi}}\frac{r^2+1}{(1+r)^2}[f_X'(1)- f_X'(0)]^2\\
	&=\frac{1}{2n\sqrt{\pi t}}+t^{3/2}\frac{A(r)}{3}\left[f_X'(1)- f_X'(0)\right]^2.
	\end{split}
	\end{equation}
	Hence, the square of the asymptotically optimal bandwidth is
	$$
	t^*=(2n\sqrt{\pi}A(r))^{-1/2}\left|f_X'(1)- f_X'(0)\right|^{-1}
	$$
	with the minimum value
	$$
	\min_t\mathrm{AMISE}\{f\}(t)=\frac{2^{5/4}\sqrt{\left|f_X'(1)- f_X'(0)\right|}}{3\pi^{3/8}}A(r)^{1/4}n^{-3/4}.
	$$
\end{enumerate}
\end{corollary}

A few remarks are in order. First, it is interesting to note that in the case of $f_X'(1)=f_X'(0)$, the optimum choice $t^*$ and the minimum AMISE do \textit{not} depend on $r$, and are the same as the more familiar `whole line' situation --- in other words, we can confidently use existing methods in the literature (such as recommended above) to choose a stopping time. Second, it seems plausible that we could estimate $f_X'(1)- f_X'(0)$ (or the value of $r$) \textit{adaptively} and change the boundary conditions in the model (\ref{eq:pde:candidate}) accordingly. A full discussion of solving the heat equation with linked boundary conditions for the first spatial derivative is beyond the scope of this paper but can be done using the same methods we present here. Future work will aim to incorporate an adaptive estimate of the true boundary conditions (both for the density function and its first derivative - we do this for the density function in \S \ref{synth_num_exam}) and resulting adaptive boundary conditions. We mention a result in this direction which will appear when we compare our model to that of \cite{Botev2010}, whose code is based around the discrete cosine transform, the continuous version of which solves the heat equation subject to the boundary conditions
$
f'_c(0)=f_c'(1)=0.
$
We have used the subscript $c$ to avoid confusion with our solution $f$ to (\ref{eq:pde:candidate}). The analogous result to Theorem \ref{AMISE_1} is the following theorem which can be proven using the same techniques and hence we have omitted the proof. Similarly, one can then derive the optimum choice of $t$ and the minimum AMISE $\mathcal{O}(n^{-3/4})$ (slower rate) under the condition that $(f_X'(1),f_X'(0))\not=(0,0)$.

\begin{theorem}[Boundary Effects on Asymptotic Bias]
\label{AMISE_2}
Let $t_n$ be such that $\lim_{n\rightarrow\infty}t_n=0$ and also $\lim_{n\rightarrow\infty}n\sqrt{t_n}=\infty$. Suppose that $f_X\in C^2([0,1])$. Then the following hold as $n\rightarrow\infty$:
\begin{enumerate}
	\item The integrated variance has the asymptotic behavior
	\begin{equation}
	\label{Var_asym_2}
	\int_0^1\mathrm{Var}_{f_X}[f_c(x,t_n)]dx\sim\frac{1}{2n\sqrt{\pi t_n}}.
	\end{equation}
	\item If $f_X'(0)=f_X'(1)=0$ then
	\begin{equation}
	\label{Exp_asym_2}
	\int_0^1\{\mathbb{E}_{f_X}[f_c(x,t_n)]-f_X(x)\}^2dx\sim t_n^2\int_0^1\frac{1}{4}\left[f_X^{''}(x)\right]^2dx.
	\end{equation}
\item 	If $f_X'(0)\neq 0$ or $f_X'(1)\neq 0$ then
	\begin{equation}
	\label{Exp_asym2_2}
	\int_0^1\{\mathbb{E}_{f_X}[f_c(x,t_n)]-f_X(x)\}^2dx\sim t_n^{3/2}\frac{4-2\sqrt{2}}{3\sqrt{\pi}}[f_X'(1)^2+f_X'(0)^2].
	\end{equation}
\end{enumerate}
\end{theorem}

\section{Numerical Approximations of the PDE Estimator}

Before giving numerical examples with the new estimator, we consider practical methods for solving the PDE \eqref{eq:pde:candidate}, in order to evaluate the KDE $
f(x,t)=\frac{1}{n}\sum_{k=1}^nK(r;x,X_k,t),$ on a regular grid. There are two different practical computational methods to compute the density estimator based on the PDE \eqref{eq:pde:candidate}:
\begin{description}
 \item[1. Series Expansion:] Essentially solving the continuous model \eqref{eq:pde:candidate} via the series or contour integral representation in Theorem~\ref{rep_thm}. 
\item[2. Backward Euler method:] Solving a discretized or binned version of \eqref{eq:pde:candidate}, as explained in the rest of this section. In Theorem~\ref{thm:discrete:model:converges}, we show that this binned estimator converges to the continuous PDE estimator.
\end{description}
The two methods have relative advantages and disadvantages. The backward Euler method is a first order finite difference method (however, this is not a problem in practice as argued below), but it is simple and easy to use, especially if the initial data is already discretely binned. The backward Euler method also maintains the key property of positivity and satisfies the same maximum principle properties as the continuous solution (see Appendix~\ref{4 corners matrix} and Lemma~\ref{thm:discrete_bound}). The reason for not using second order methods such as Crank--Nicolson is that for large time steps this would not preserve non-negativity of the solution. In other words, the discrete solution can no longer be interpreted as a probability distribution (a well-known result says that any general linear method that is unconditionally positivity preserving for all positive ODEs must have order $\leq 1$ \cite{bolley1978conservation}). However, methods such as Crank--Nicolson can also easily be used for the discrete model if desired, but for brevity we do not discuss such methods further. The series expansion of the continuous PDE model is typically highly accurate for $t>0$, but less easy to implement. We provide MATLAB codes for both methods:\\\noindent{}\url{https://github.com/MColbrook/Kernel-Density-Estimation-with-Linked-Boundary-Conditions}.

To derive the appropriate time-stepping method, we do the following:
\begin{enumerate}
\item We approximate the exact solution $f$ by a vector $\v u$. That is, $u(x_i; \cdot) \approx  f(x_i,\cdot)$. Here $x_i = i h$ is the $i$th grid point on the grid of $m+2$ equally spaced points in the domain $[0,1]$, for $i=0,1,\ldots,m,m+1$.
 The spacing between two consecutive grid points is
$
h = \frac{1}{m+1}.
$
Note here that $m$ is typically smaller than $n$, the number of samples that form the empirical measure.
\item The two boundary conditions in \eqref{eq:pde:candidate} give two equations involving values at the two boundary nodes, i.e. at node $0$ and at node $m+1$.
That is,
  \begin{eqnarray}
 \quad  u_0 &=& r u_{m+1}, \\
u_1 - u_0 &=& u_{m+1} - u_{m} .
 \end{eqnarray}
This motivates us to make the following \textit{definitions} for the boundary nodes:
 \begin{equation}
 u_0 \idef \frac{r}{r+1} (u_1+u_m) , \qquad u_{m+1} \idef \frac{1}{r+1} (u_1+u_m) .
 \label{eq:def:ghost:nodes}
 \end{equation}
We are left with a set of $m$ equations involving $m$ unknown values $u_1, \ldots, u_m$, at the $m$ interior nodes $1, \ldots, m$, where we use a standard second-order finite difference approximation of the (spatial) second derivative.
\item 
We consider the corresponding $m \times m$ \emph{four-corners matrix} with the following structure:
\begin{equation}
\m A =
\left(
\begin{tabular}{ccccc}
$2-\frac{r}{r+1}$ & $-1$   & &   & $-\frac{r}{r+1}$ \\
 $-1$ & $2$ & $-1$ \\
  & $\ddots$ & $\ddots$ & $\ddots$ \\
 & & $-1$ & $2$ & $-1$ \\
 $-\frac{1}{r+1}$ &  & &   $-1$ & $-\frac{1}{r+1}$
\end{tabular}
\right).
\label{eq:four:corners:matrix}
\end{equation}
\end{enumerate}

Given a time $T$ at which we wish to evaluate the solution, we consider a time step $\Delta t =2h^2$. For ease of the analysis, we assume that $T$ is a multiple of $\Delta t$, though his can be avoided by making the last time step smaller if needed. We use a superscript $k$ to denote the solution at time $k\Delta t$ (i.e. the $k$th step), then the backwards Euler method can be written as
\begin{equation}
\label{backwards_euler}
\v u^{k+1}=\left(\m I+\m A\right)^{-1}\v u^k, \quad k=0,..., T/\Delta t -1,
\end{equation}
where $\m I$ denotes the $m\times m$ identity matrix. The matrix inverse can be applied in $\mathcal{O}(m)$ operations using the fact that $\m A$ is a rank one perturbation of a tridiagonal matrix. Even though we take small time steps, the total time $T=\mathcal{O}(n^{-2/5})$ is small. It follows that the total complexity is $\mathcal{O}(m^3n^{-2/5})$, giving an error (in the interior) of order $\mathcal{O}(h^2)=\mathcal{O}(m^{-2})$. The error of the continuous model scales as $\mathcal{O}(n^{-2/5})$. If there is freedom in selecting the number of bins $m+2$, this suggests choosing $m=\mathcal{O}(n^{1/5})$ which leads to a modest $\mathcal{O}(n^{1/5})=\mathcal{O}(m)$ complexity. A key property of the matrix \eqref{eq:four:corners:matrix} is that it has zero column sum, off-diagonals are negative or zero, and the main diagonal entries are positive. This allows the interpretation of \eqref{backwards_euler} as a \textit{discrete-time Markov process}. In Appendix \ref{4 corners matrix}, we prove the following theorem for completeness (using explicit formulae for the eigenvalues and eigenvectors of $\m A$).
 \begin{theorem}[Convergence of Binned to Diffusion Estimator]
\label{thm:discrete:model:converges} 
The solution of the binned estimator \eqref{backwards_euler} with the four corner matrix in \eqref{eq:four:corners:matrix} converges to the solution of the continuous problem \eqref{eq:pde:candidate} as $m \rightarrow \infty$:
\[
\sup_{\epsilon\leq t\leq T}\sup_{0\leq k\leq m+1}|u(k/(m+1);t)-f(k/(m+1);t)|\rightarrow 0,\qquad n \rightarrow \infty.
\]
\end{theorem}
 
\begin{proof}
See Appendix \ref{4 corners matrix}.
\end{proof}



Further interesting properties of the discrete system are discussed in Appendix \ref{4 corners matrix}. In Theorem \ref{thm:discrete:model:converges}, we have restricted $t\geq \epsilon>0$ to include the possibility that the initial condition may not be a proper function, but an empirical measure. We finally remark that sometimes the solution is needed at later times (e.g. $\mathcal{O}(1)$), for example when querying the solution at various times $t$ as part of minimizing least squares cross validation to determine a good choice of $T$. In that case, we recommend computing the matrix exponential
\[\textstyle
 \v u(t) = \exp \left(-\frac{t}{2h^2} \m A  \right) \v u(0).
 \]
There are many possible methods to compute the matrix exponential \cite{moler2003nineteen}, such as MATLAB's \texttt{expm} code based on \cite{higham2005scaling,al2010new}.

\section{Numerical Experiments}

\subsection{Numerical examples with synthetic data}
\label{synth_num_exam}

First, we test the estimator on examples where the true density $f_X$ is known. We begin with the trimodal distribution shown in Figure \ref{syn1a}. We will demonstrate two versions of the method. First, when the exact value of $r$ is known (labelled ``Linked 1''), and second where we estimate the value of $r$ by
$$
r_{\mathrm{est}}=\frac{\sum_{j=1}^n\chi_{<n^{-1/2}}(X_j)}{\sum_{j=1}^n\chi_{>1-n^{-1/2}}(X_j)}
$$
(labelled ``Linked 2''). We expect both to perform similarly for sufficiently large $n$. For stopping times, we have used the software that adaptively chooses the bandwidth, discussed in Section \ref{asymp_properties}. In other words, we do not give our algorithms any information other than the given sample. We compare with three other methods. The first is the density estimation proposed in \cite{Botev2010} based on the discrete cosine transform (labelled ``Cosine''). The second is the well-known and arguably state-of-the-art beta kernel method of \cite{chen1999beta}, which we label ``Beta'' in the plots. This method is free from boundary bias, at the cost of an increased boundary variance. Finally, we also compare with a method which uses copula kernels \cite{jones2007kernel} and which has been found to be competitive with the beta kernel approach of \cite{chen1999beta}. This method has an automatic bandwidth selector which we shall use, and we label it ``Copula'' in the plots. The latter two methods are freely available in the R package \texttt{evmix} \cite{EV} which can be found at \url{https://CRAN.R-project.org/package=evmix}.

We estimate the error using the $L^2$ and $L^\infty$ norms at the points $l\times 10^{-3}$ for $l=0,...,10^3$. The only change is when considering the copula method, where we take $l=1,...,10^3-1$ instead since we found this method to be unstable near the boundaries. Figure \ref{syn1a} shows a typical approximation of the distribution function using our proposed method and the other methods for a sample size of $n=10^4$. Our proposed method is more accurate near the boundaries of the domain (see magnified section of plots) and behaves similarly in the middle of the domain. We found that using the estimate $r_{\mathrm{est}}$ instead of the exact value of $r$ did not have a great effect on the error. In other words, we can apply our model without needing to know the value of $r$.

Figure \ref{syn1b} (left) shows the $L^2$ measure of error averaged over $100$ independent samples for each $n$. The $L^2$ errors for both ``Linked'' methods and the ``Cosine'' method agreed almost perfectly with the minimum AMISE and the analysis in Section~\ref{asymp_properties} for large $n$. Using our model with an estimate of $r$ increases the convergence rate from $\mathcal{O}(n^{-3/4})$ to $\mathcal{O}(n^{-4/5})$. Both ``Linked'' methods and the ``Cosine'' method are found to be more accurate than the ``Beta'' and ``Copula'' methods. The tailing-off convergence for the ``Copula'' method was due to a need to implement a lower bound for the bandwidth. Below this limit, we found the ``Copula'' method to be unstable. Figure \ref{syn1b} (right) shows the same plot but now for the $L^\infty$ measure of error. Here we see a more pronounced difference between the methods, with both ``Linked'' methods producing much smaller errors than the other methods. We found the same behavior in these plots for a range of other tested distributions. Finally, we comment on the CPU times for each method, shown in Figure \ref{syn1c} (averaged over the 100 samples for each $n$). In order to produce a fair comparison, we have included the CPU time taken for automatic bandwidth selection when using the ``Linked'' methods. All methods appear to have CPU times that grow linearly with $n$. The ``Cosine'' method in fact scales like $\mathcal{O}(n\log(n))$ due to the use of the discrete cosine transform. The linked estimator is faster by about an order of magnitude than the other methods. This is due to the exponential decay of the series for $t>0$ - only a small number of terms need to be summed in order to get very accurate results.

\begin{figure}
\centering
\includegraphics[width=0.495\textwidth,trim={32mm 95mm 35mm 95mm},clip]{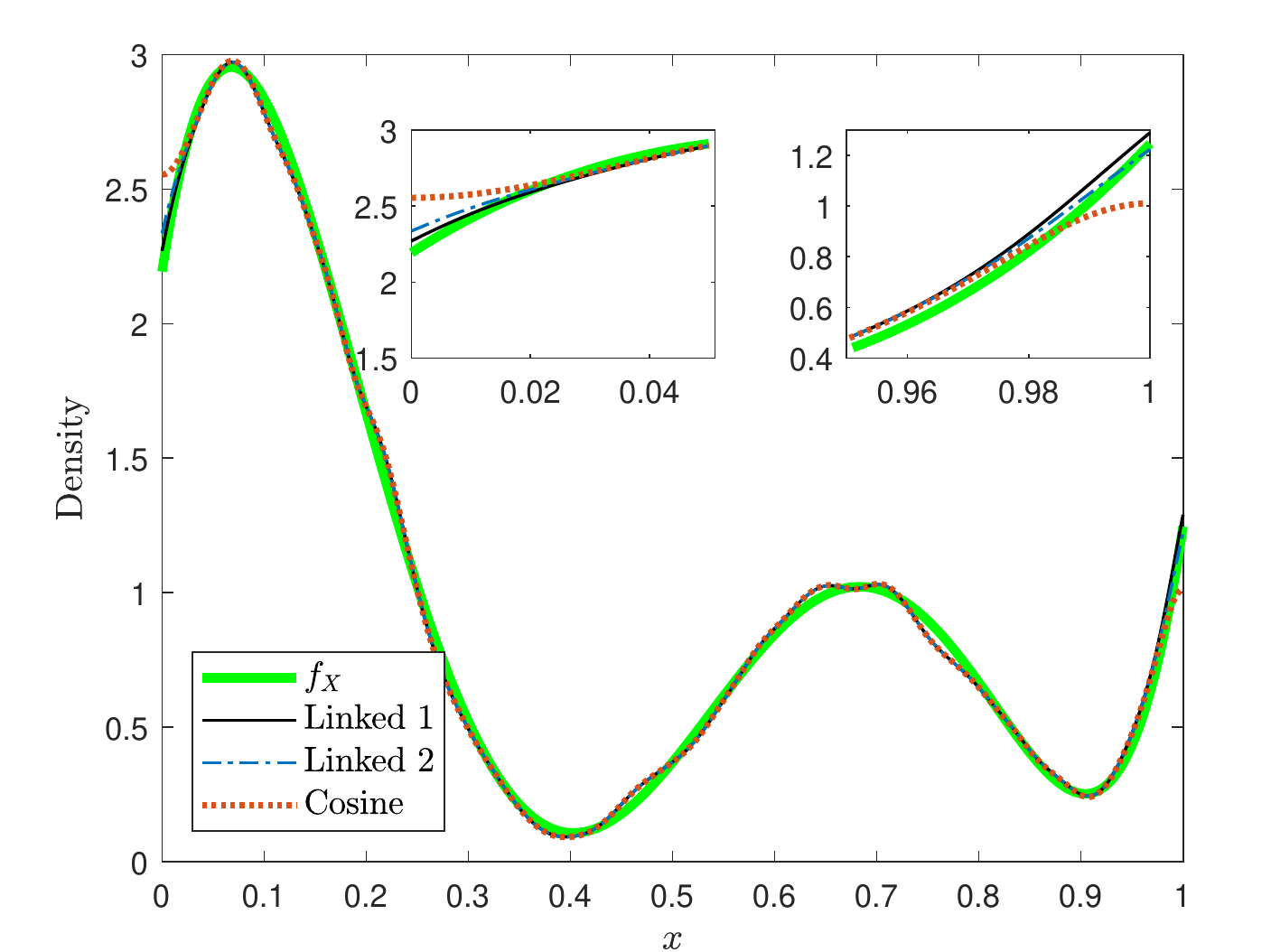}
\includegraphics[width=0.495\textwidth,trim={32mm 95mm 35mm 95mm},clip]{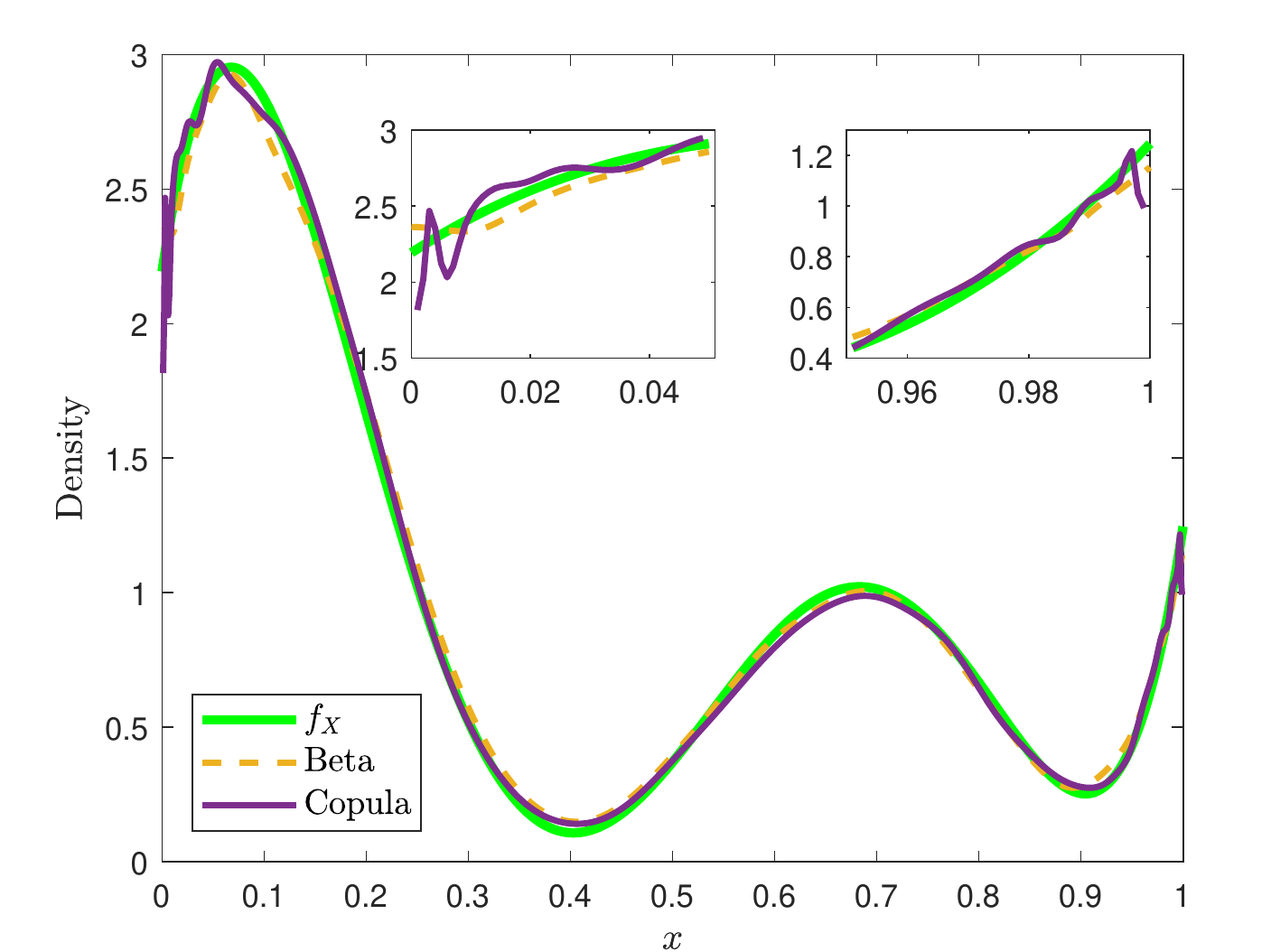}
\caption{Example of different methods for a sample size $n=10^4$. The proposed diffusion model (``Linked'') is much more accurate near the boundaries than the cosine model (``Cosine'') as highlighted by the magnified sections. The method ``Copula'' is found to be unstable near the boundaries.}
\label{syn1a}
\end{figure}

\begin{figure}
\centering
\includegraphics[width=0.495\textwidth,trim={32mm 93mm 35mm 95mm},clip]{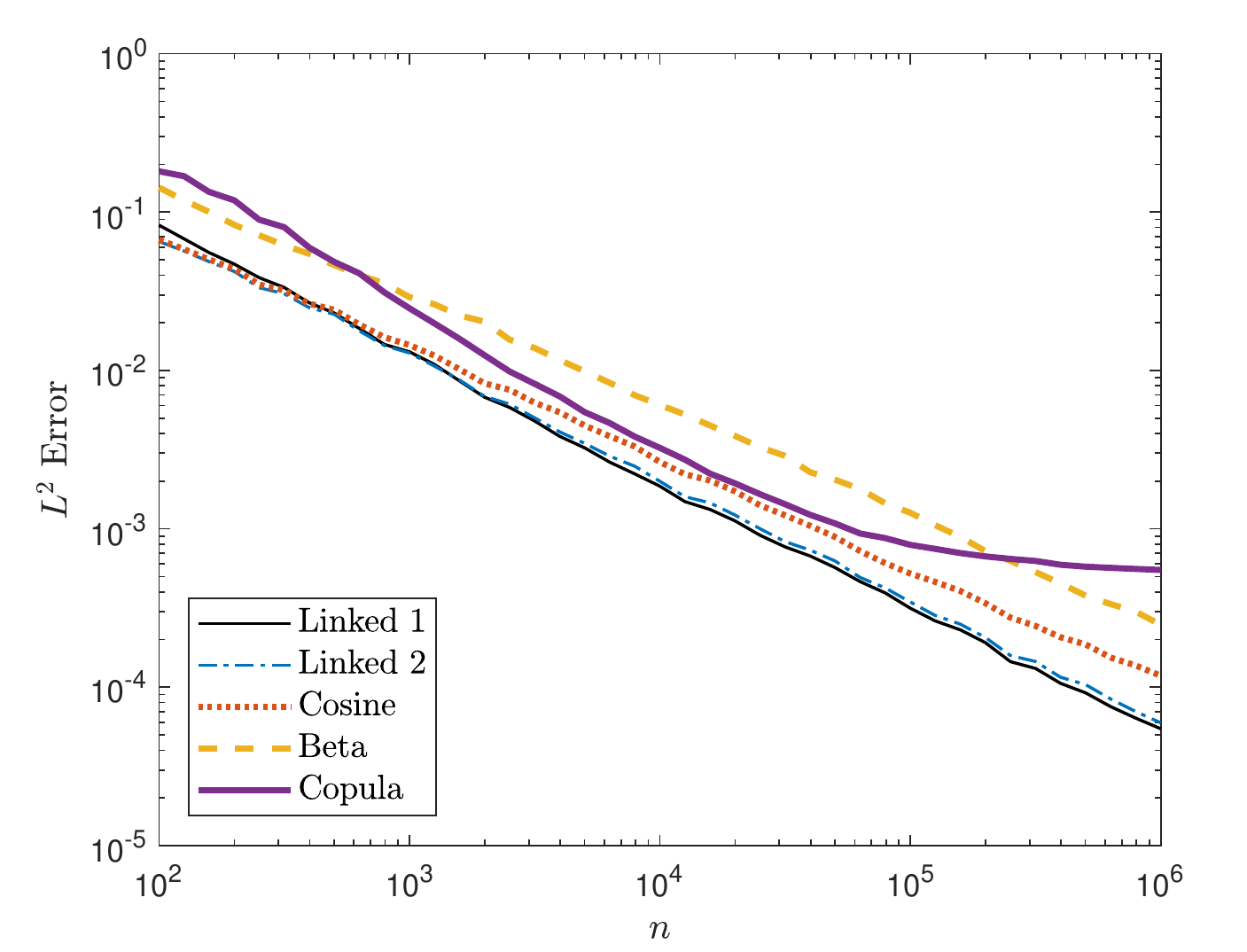}
\includegraphics[width=0.495\textwidth,trim={32mm 93mm 35mm 95mm},clip]{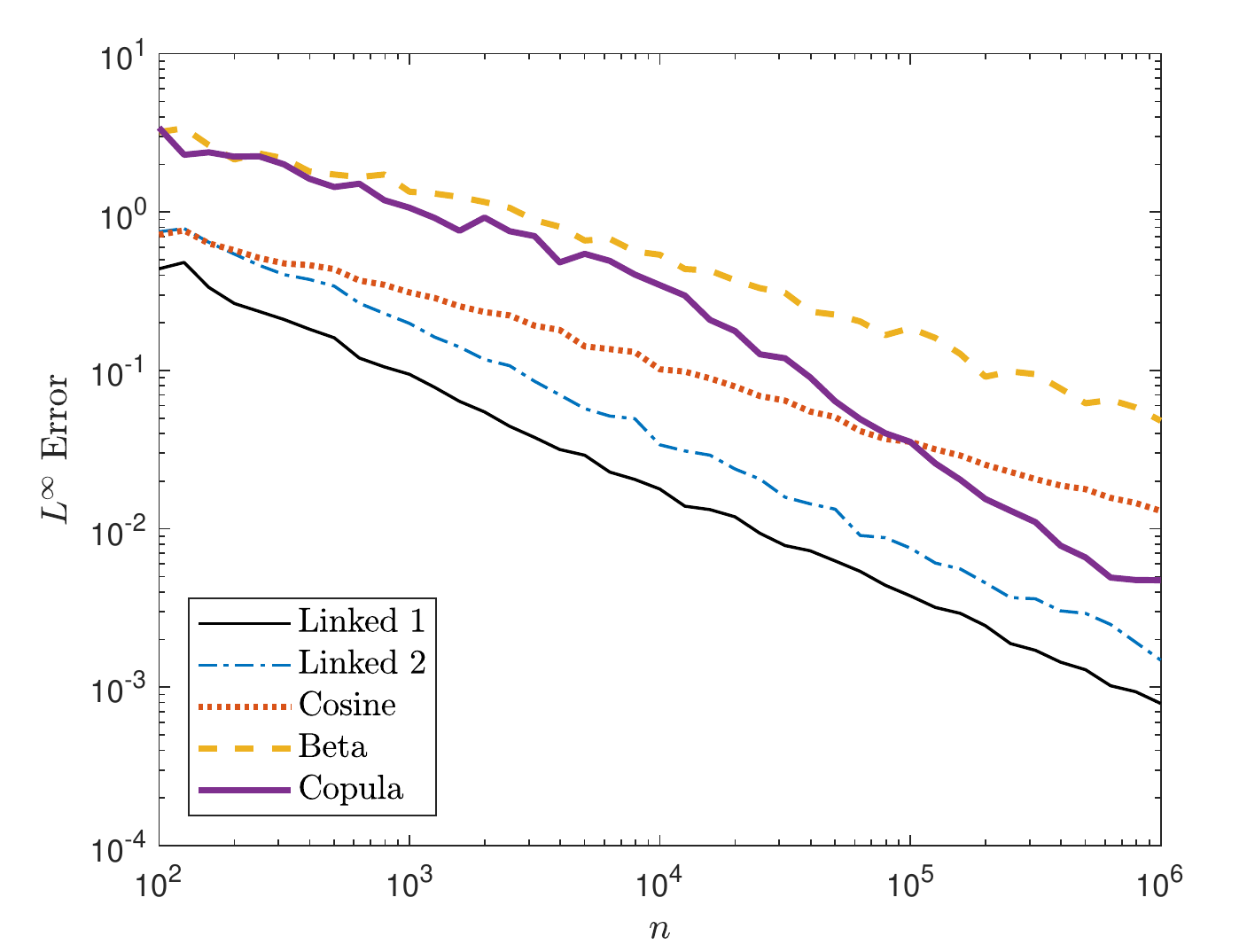}
\caption{Left: $L^2$ errors of methods averaged over $100$ samples for each $n$. Right: $L^\infty$ errors of methods averaged over $100$ samples for each $n$. The $L^2$ errors agree well with the minimum AMISE from Section \ref{asymp_properties}, whereas the increased accuracy gained near the boundary by using the linked boundary model is highlighted by the $L^\infty$ errors.}
\label{syn1b}
\end{figure}

\begin{figure}
\centering
\includegraphics[width=0.5\textwidth,trim={32mm 93mm 32mm 95mm},clip]{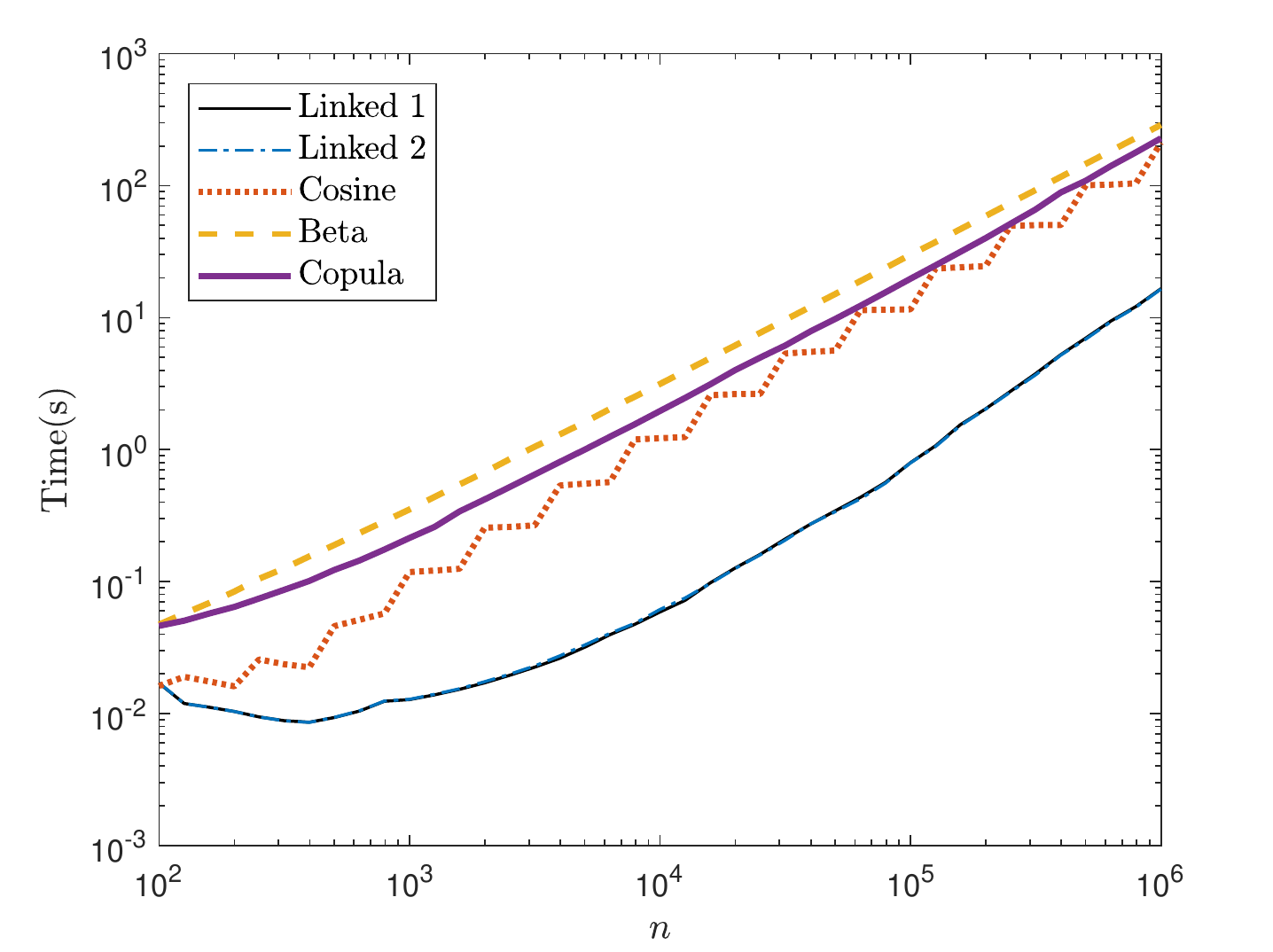}
\caption{CPU times for each method averaged over $100$ samples for each $n$. Experiments were performed on a basic four year old laptop. Each method appears to grow almost linearly (up to logarithmic factors), with the linked boundary estimator an order of magnitude faster than the other methods.}
\label{syn1c}
\end{figure}

\begin{table}
\centering
\begin{tabular}{c|llllllllll}
\toprule
\multicolumn{1}{c}{$a$}             & \multicolumn{1}{c}{$1.1$}                & \multicolumn{1}{c}{$1.2$} & \multicolumn{1}{c}{$1.3$} & \multicolumn{1}{c}{$1.4$} & \multicolumn{1}{c}{$1.5$}  \\
\midrule \midrule
\multicolumn{1}{l}{\textbf{Linked}} & \num{2.98E-3} & \num{1.33E-3} & \colorbox[rgb]{0.75,0.75,0.75}{\num{6.82E-4}} & \colorbox[rgb]{0.75,0.75,0.75}{\num{3.22E-4}} & \colorbox[rgb]{0.75,0.75,0.75}{\num{2.38E-4}}            \\
\midrule
\multicolumn{1}{l}{\textbf{LC}}  &  \colorbox[rgb]{0.75,0.75,0.75}{\num{1.05E-3}} & \num{1.14E-3} & \num{1.26E-3} & \num{1.38E-3} & \num{1.52E-3}  \\
\midrule
\multicolumn{1}{l}{\textbf{LCS}}                      & \num{1.23E-3} & \colorbox[rgb]{0.75,0.75,0.75}{\num{1.03E-3}} & \num{9.42E-4}  &  \num{1.04E-3}                            &     \num{1.19E-3}                                                                    \\
\bottomrule
\end{tabular}
~\\
~\\
\begin{tabular}{c|llllllllll}
\toprule
\multicolumn{1}{c}{$a$}              & \multicolumn{1}{c}{$1.6$} & \multicolumn{1}{c}{$1.7$} & \multicolumn{1}{c}{$1.8$} & \multicolumn{1}{c}{$1.9$} & \multicolumn{1}{c}{$2$} \\
\midrule \midrule
\multicolumn{1}{l}{\textbf{Linked}} & \colorbox[rgb]{0.75,0.75,0.75}{\num{1.58E-4}}          & \colorbox[rgb]{0.75,0.75,0.75}{\num{1.13E-4}}          & \colorbox[rgb]{0.75,0.75,0.75}{\num{8.01E-5}}         & \colorbox[rgb]{0.75,0.75,0.75}{\num{5.96E-5}}          & \colorbox[rgb]{0.75,0.75,0.75}{\num{5.05E-5}}           \\
\midrule
\multicolumn{1}{l}{\textbf{LC}}   &\num{1.65E-3}  & \num{1.80E-3} &   \num{1.94E-3} & \num{2.09E-3}  &   \num{2.27E-3} \\
\midrule
\multicolumn{1}{l}{\textbf{LCS}}                 &  \num{1.30 E-3}                              &       \num{1.40 E-3}                         &  \num{1.66 E-3}                              &      \num{1.74E-3}                          & \num{2.16E-3}                                                              \\
\bottomrule
\end{tabular}
\caption{Mean $L^2$ squared error over 10 simulations for different $a$.}
\label{tab1}
\end{table}

\begin{table}
\centering
\begin{tabular}{c|llllllllll}
\toprule
\multicolumn{1}{c}{$a$}             & \multicolumn{1}{c}{$1.1$}                & \multicolumn{1}{c}{$1.2$} & \multicolumn{1}{c}{$1.3$} & \multicolumn{1}{c}{$1.4$} & \multicolumn{1}{c}{$1.5$}  \\
\midrule \midrule
\multicolumn{1}{l}{\textbf{Linked}} & \colorbox[rgb]{0.75,0.75,0.75}{\num{7.32E-2}} & \colorbox[rgb]{0.75,0.75,0.75}{\num{4.19E-2}} & \colorbox[rgb]{0.75,0.75,0.75}{\num{2.52E-2}} & \colorbox[rgb]{0.75,0.75,0.75}{\num{1.31E-2}}  & \colorbox[rgb]{0.75,0.75,0.75}{\num{7.97E-3}}    \\
\midrule
\multicolumn{1}{l}{\textbf{LC}}                      &   \num{5.34E-1}                                            &     \num{6.40E-1}                           &\num{7.51E-1}                                &         \num{8.71E-1}                       &   \num{1.00E0}     \\
\midrule
\multicolumn{1}{l}{\textbf{LCS}}                      &   \num{1.84E-1}                                            &   \num{1.26E-1}                             &   \num{1.42E-1}                             &     \num{1.72E-1}                           &    \num{1.95E-1}       \\
\bottomrule
\end{tabular}
~\\
~\\
\begin{tabular}{c|llllllllll}
\toprule
\multicolumn{1}{c}{$a$}              & \multicolumn{1}{c}{$1.6$} & \multicolumn{1}{c}{$1.7$} & \multicolumn{1}{c}{$1.8$} & \multicolumn{1}{c}{$1.9$} & \multicolumn{1}{c}{$2$} \\
\midrule \midrule
\multicolumn{1}{l}{\textbf{Linked}}   & \colorbox[rgb]{0.75,0.75,0.75}{\num{4.42E-3}} & \colorbox[rgb]{0.75,0.75,0.75}{\num{2.85E-3}}          & \colorbox[rgb]{0.75,0.75,0.75}{\num{1.18E-3}}         & \colorbox[rgb]{0.75,0.75,0.75}{\num{4.78E-4}}          & \colorbox[rgb]{0.75,0.75,0.75}{\num{2.39E-4}}          \\
\midrule
\multicolumn{1}{l}{\textbf{LC}}                                            &            \num{1.14E0}                    &       \num{1.28E0}                         &         \num{1.44E0}                       &   \num{1.60E0}                             &             \num{1.78E0}                    \\
\midrule
\multicolumn{1}{l}{\textbf{LCS}}                                            &           \num{2.27E-1}                     &    \num{2.47E-1}                            &           \num{2.82E-1}                     &    \num{3.22E-1}                            &      \num{3.70E-1}                           \\
\bottomrule
\end{tabular}
\caption{Mean $L^\infty$ squared error over 10 simulations for different $a$.}
\label{tab2}
\end{table}

\begin{figure}[t]
\centering
\includegraphics[width=0.49\textwidth,trim={32mm 95mm 35mm 95mm},clip]{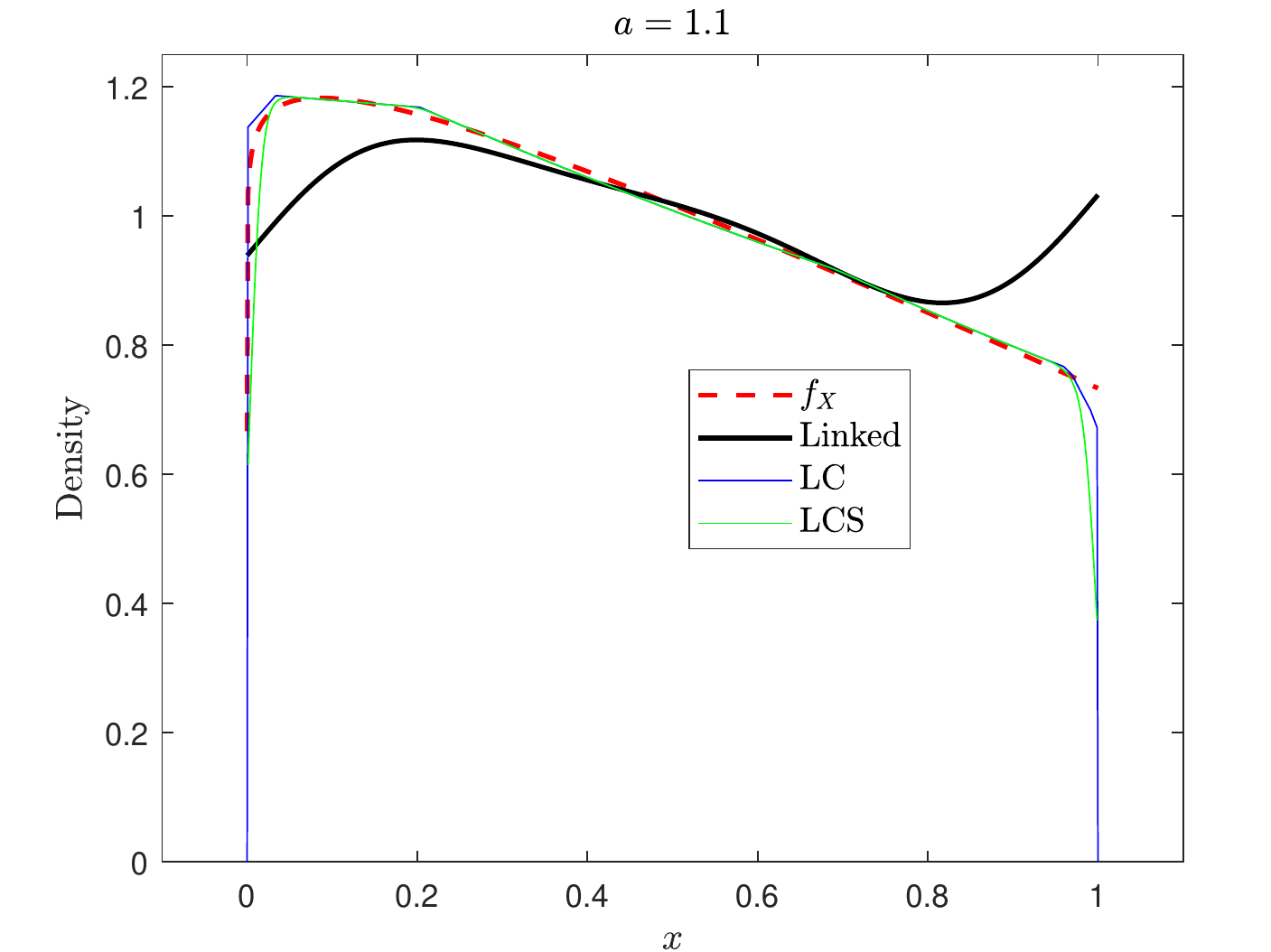}
\includegraphics[width=0.49\textwidth,trim={32mm 95mm 35mm 95mm},clip]{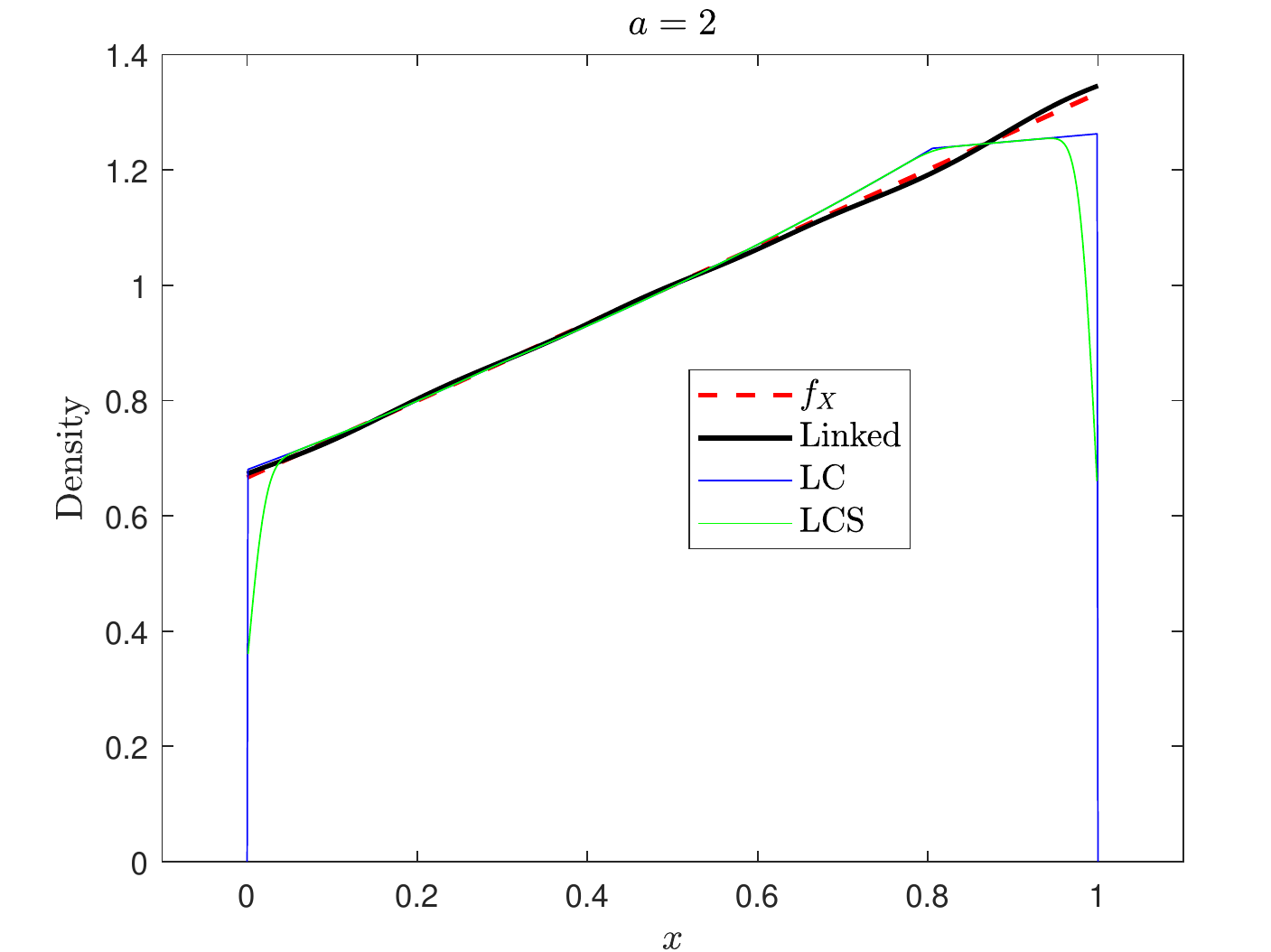}
\caption{Typical estimates for $n=10^4$ and $a=1.1$, $a=2$. We used the R package \texttt{logcondens} for the log-concave projection method.}
\label{syn2}
\end{figure}

Next, we consider the case when $f_X$ is log-concave and not necessarily smooth. Denoting the PDF of the beta distribution with parameters $(\alpha,\beta)$ by $b(\alpha,\beta;x)$, we let
$$
f_X(x)=\frac{b(1,2;x)+2b(a,1;x)}{3}.
$$
The parameter $a$ controls the smoothness of $f_X$ near $x=0$. We have compared our method to a method that computes log-concave maximum likelihood estimators \cite{dumbgen2007active,dumbgen2010logcondens}. This seeks to compute the log-concave projection of the empirical distribution through an active set approach. Code is freely available in \texttt{logcondens} \cite{logR} which can be found at \url{https://CRAN.R-project.org/package=logcondens}. Details on such methods can be found in \cite{samworth2017recent}, with a study of the more involved case of censored data in \cite{dumbgen2014maximum}. Tables \ref{tab1} and \ref{tab2} show the mean squared $L^2$ and $L^\infty$ errors respectively over $10$ simulations for $n=10^5$, as we vary $a$ for the linked boundary diffusion estimator and the log-concave projection method (abbreviated to LC), as well as its smoothed version (LCS). In each case, we have shaded the most accurate estimator. The linked boundary diffusion estimator performs much better when measured in the uniform norm but is slightly worse in the $L^2$ sense when the distribution function becomes less smooth. This is demonstrated in Figure \ref{syn2} for a typical estimation using $n=10^4$. To produce the tables, the linked boundary diffusion estimator took about 0.5s on average per simulation, the log-concave projection took about 5s, but its smoothed version was much slower, taking about 73s.

\subsection{Numerical example with cell data}
\label{num_exam}

\begin{figure}
    \centering
    \includegraphics[scale=1.1]{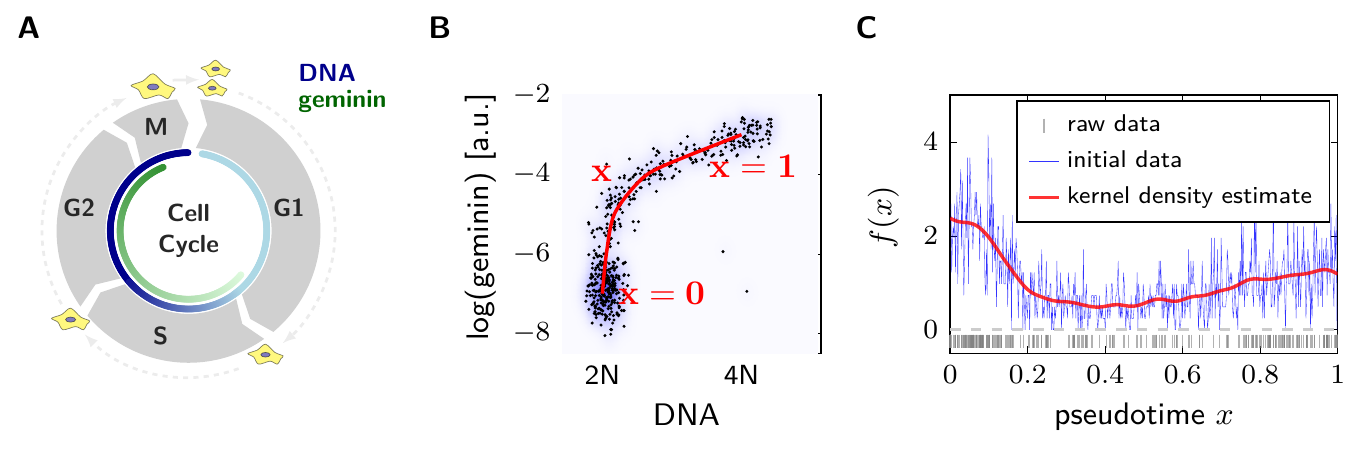}
    \caption{(A) Schematic cell cycle with geminin expression
        starting at the end of G1. (B) DNA and geminin signal from individual
        cells can be used to obtain a pseudo-temporal ordering of the
        population. An average cell follows the indicated path (red) through the
        dataset. (C) Pseudotime values (gray), binned data (blue) and kernel
        density estimate (red). The kernel density estimate was obtained by solving our continuous PDE \eqref{eq:pde:candidate} by our discrete numerical method with the `four corners matrix' in \eqref{eq:four:corners:matrix}. The stopping time, $t=0.00074$, came from the stopping time software of one of the authors: \url{https://au.mathworks.com/matlabcentral/fileexchange/14034-kernel-density-estimator}.}
    \label{fig:Example}
\end{figure}

This section demonstrates the application of the methods that we propose to a problem in biology with the data taken from \cite{Kuritz2017}. As mentioned in the introduction, Figure~\ref{fig:ksdensity:BAD} shows an example of what goes wrong when current methods are applied. Figure~\ref{fig:Example} C demonstrates our proposed method, which successfully incorporates the desired linked boundary condition.

This example originates from the study of biological processes, in particular, cell cycle studies in cancer research (Figure~\ref{fig:Example}~A). A recently developed theory \cite{Kafri2013,Kuritz2017} which relies on the distribution of cells along the cell cycle 
enables the study of entire cell cycle progression kinetics. The method utilizes data from single cell experiments like flow cytometry or single cell RNA sequencing, where the abundance of up to thousands of cellular components for every individual cell in a population is measured. Cells in a unsynchronized cell population are spread over all stages of the cell cycle, which can be seen in the exemplary dataset where levels of DNA and geminin in single cells were measured by flow cytometry (Figure~\ref{fig:Example}~B). The red curve in
Figure~\ref{fig:Example}~B indicates the path that the average cell takes when it goes through the cell cycle. Pseudotime algorithms perform a dimensionality reduction by assigning a pseudotime value to each cell, which can be interpreted as its position on the average curve. In this example, the pseudotime is a quantitative value of the progression through the cell cycle. However, it is in general not equal to real time. As the number of cells in a particular stage is related to the average transit time through that stage, one can derive a mapping from pseudotime to real time based on ergodic principles \cite{Kafri2013,Kuritz2017,Kuritz2020}. This mapping relies on the distribution of cells on the pseudotime scale. As mentioned in the introduction, the distribution at the beginning and the end of the cell cycle are linked due to cell division by
\begin{equation}
    f(0,t) = 2 \,f(1,t) \;.
    \label{eq:BC}
\end{equation}
Ignoring this fact when estimating the density on the pseudotime scale results in an erroneous transformation and thus inaccurate kinetics. The KDE with linked boundary condition ($r=2$) produces a distribution that satisfies the conditions \eqref{eq:BC} on the density due to cell division (Figure~\ref{fig:Example}~C). The MAPiT toolbox for single-cell data analysis \cite{Kuritz2020} applies our new KDE with linked boundary conditions to analyze cell cycle dependent molecular kinetics.

\section{Conclusion}
Our study was motivated by a dataset from a biological application. This biological application required a method of density estimation that can handle the situation of linked boundaries, which are crucial for gaining correct kinetics. More broadly, boundary bias issues are known to be a difficult problem in the context of kernel density estimation. To our knowledge, the linked boundary conditions that we handle here have not been previously addressed. We have proposed a new diffusion KDE that can successfully handle the linked boundary conditions. By using the unified transform, we obtained an explicit solution. In particular, we proved that this diffusion estimator is a bona fide probability density, which is also a consistent estimator at the linked boundaries, and derived its asymptotic integrated squared bias and variance (which shows an increase in the rate of convergence with sample size).

We also proposed two numerical methods to compute the estimator --- one is based on its series or integral representation and the other on the backward Euler method. We proved that the discrete/binned estimator converges to the continuous estimator. We found the new method competes well with other existing methods, including state-of-the-art methods designed to cope with boundary bias, both in terms of speed and accuracy. In particular, the new method is more accurate close to the boundary. Our new KDE with linked boundary conditions is now used in the MAPiT toolbox for single-cell data analysis \cite{Kuritz2020} to analyze cell cycle dependent molecular kinetics.

There remain some open questions regarding the proposed models. First, it is possible to adapt the methods in this paper to multivariate distributions. Second, it is possible to adapt these methods to other types of boundary conditions such as constraints on the moments of the distribution (and other non-local conditions). In this regard, we expect that the flexibility of the unified transform in PDE theory will be useful in designing smoothing kernel functions with the desired statistical properties. 

\paragraph{Acknowledgments \& Contributions:} 

MJC was supported by EPSRC grant EP/L016516/1. ZIB was supported by ARC grant DE140100993. KK was supported by DFG grant AL316/14-1 and by the Cluster of Excellence in Simulation Technology (EXC 310/2) at the University of Stuttgart. SM was supported by the ARC Centre of Excellence for Mathematical and Statistical Frontiers (ACEMS). MJC performed the theoretical PDE/statistical analysis of both the continuous and discrete models, and the numerical tests. SM developed and tested the binned version of the estimator. ZIB proposed the PDE model and assisted MJC and SM in writing the paper. KK provided the cell data and assisted in the writing of the numerical section. MJC is grateful to Richard Samworth, Tom Trogdon and David Smith for comments, and to Arieh Iserles for introducing him to the problem. The authors are grateful to the referees for comments that improved the manuscript.

\appendix

\section{Proofs of Results in Section \ref{cts_model_properties}}


\subsection{Formal derivation of solution formula}
\label{sol_deriv}
We begin with a formal description of how to obtain the solution formulae in Theorem \ref{rep_thm}. The most straightforward way to construct the solution is via the unified transform, and the following steps provide a formal solution which we must then rigorously prove is indeed a solution.

The first step is to write the PDE in divergence form:
$$
[\exp(-ikx+k^2t/2)f]_t-\frac{1}{2}[\exp(-ikx+k^2t/2)(f_x+ikf)]_x=0, \quad k\in\mathbb{C}.
$$
We will employ Green's theorem,
\begin{equation} \textstyle
\iint_{\Omega}\Big(\frac{\partial F}{\partial x}-\frac{\partial G}{\partial y}\Big) dxdy=\int_{\partial\Omega}\big(Gdx+Fdy\big),
\end{equation}
over the domain $(0,1)\times(0,t)$. Here one must assume apriori estimates on the smoothness of the solution $f$ which will be verified later using the candidate solution. Define the transforms:
\begin{align*}
&\textstyle\hat{f}_0(k):=\int_0^1 \exp(-ikx)f_0(x)dx,\quad &&\textstyle\hat{f}(k,t):=\int_0^1 \exp(-ikx)f(x,t)dx,\\
&\textstyle\tilde{g}(k,t):=\int_0^t\exp(k\tau)f(1,\tau)d\tau,\quad &&\textstyle\tilde{h}(k,t):=\int_0^t\exp(k\tau)f_x(1,\tau)d\tau,
\end{align*}
where again we assume these are well defined. Green's theorem and the boundary conditions imply (after some small amount of algebra) the so called `global relation', coupling the the transforms of the solution and initial data:
\begin{equation}
\label{GR}
\begin{split}
\hat{f}(k,t)\exp(k^2t/2)=&\textstyle\hat{f}_0(k)-\frac{1}{2}[\tilde{h}(k^2/2,t)+ikr\tilde{g}(k^2/2,t)]\\
&+\textstyle\frac{\exp(-ik)}{2}[\tilde{h}(k^2/2,t)+ik\tilde{g}(k^2/2,t)], \quad k\in\mathbb{C}.
\end{split}
\end{equation}
The next step is to invert via the inverse Fourier transform, yielding
\begin{equation}
\begin{split}
f(x,t)=\textstyle\frac{1}{2\pi}\int_{-\infty}^\infty &\exp(ikx-k^2t/2)\big\{\hat{f}_0(k)-\frac{1}{2}[\tilde{h}(k^2/2,t)+ikr\tilde{g}(k^2/2,t)]\\
&+\textstyle\frac{\exp(-ik)}{2}[\tilde{h}(k^2/2,t)+ik\tilde{g}(k^2/2,t)]\big\}dk.
\end{split}
\end{equation}
However, this expression contains the unknown functions $\tilde g$ and $\tilde h$. To get rid of these we use some complex analysis and symmetries of the global relation (\ref{GR}). Define the domains
\begin{equation}
D^+=\{k\in\mathbb{C}^+:\mathrm{Re}(k^2)<0\},\quad D^-=\{k\in\mathbb{C}^-:\mathrm{Re}(k^2)<0\},\quad D=D^+\cup D^{-}. 
\end{equation}
These are shown in Figure \ref{domains}.
A quick application of Cauchy's theorem and Jordan's lemma means we can re-write our solution as
\begin{equation}
\label{sol1}
\begin{split}
f(x,t)=&\textstyle\frac{1}{2\pi}\int_{-\infty}^\infty \exp(ikx-k^2t/2)\hat{f}_0(k)dk\\
&\textstyle-\frac{1}{2\pi}\int_{\partial D^+}\frac{\exp(ikx-k^2t/2)}{2}[\tilde{h}(k^2/2,t)+ikr\tilde{g}(k^2/2,t)]dk\\
&\textstyle-\frac{1}{2\pi}\int_{\partial D^-}\frac{\exp(ik(x-1)-k^2t/2)}{2}[\tilde{h}(k^2/2,t)+ik\tilde{g}(k^2/2,t)]dk.
\end{split}
\end{equation}
We now use the symmetry under $k\rightarrow-k$ of the global relation (\ref{GR}) and the fact that the argument in each of $\tilde{g}$ and $\tilde{h}$ is $k^2/2$ to set up the linear system:
\begin{equation*}
\begin{split}
\frac{1}{2}
\begin{pmatrix} 
[\exp(-ik)-1] & ik[\exp(-ik)-r] \\
[\exp(ik)-1] & -ik[\exp(ik)-r]
\end{pmatrix}\begin{pmatrix} 
\tilde{h}(\frac{ k^2}{2},t) \\
\tilde{g}(\frac{ k^2}{2},t) 
\end{pmatrix}=\begin{pmatrix} 
\hat{f}(k,t)\exp(\frac{ t k^2}{2})-\hat{f}_0(k)\\
\hat{f}(-k,t)\exp(\frac{t k^2}{2})-\hat{f}_0(-k)
\end{pmatrix}.
\end{split}
\end{equation*}
Defining the determinant function $\Upsilon(k)=2(1+r)(\cos(k)-1),$ solving the linear system leads to the relations:
\begin{align*}\textstyle
\frac{\tilde{h}(k^2,t)+ikr\tilde{g}(k^2,t)}{2}&=\textstyle\frac{1}{\Upsilon(k)}\Big\{\hat{f}_0(k)[(1+r)\exp(ik)-2r]\\
&\quad\quad\quad\quad+\textstyle\hat{f}_0(-k)(1-r)\exp(-ik) \\
&\quad\quad\quad\quad\quad-\textstyle\exp(k^2t/2)\hat{f}(k,t)[(1+r)\exp(ik)-2r]\\
&\quad\quad\quad\quad\quad\quad-\textstyle\exp(k^2t/2)\hat{f}(-k,t)(1-r)\exp(-ik)\Big\},\\
\textstyle\frac{\tilde{h}(k^2/2,t)+ik\tilde{g}(k^2/2,t)}{2}&=\frac{1}{\Upsilon(k)}\Big\{\hat{f}_0(k)[2\exp(ik)-(1+r)]+\hat{f}_0(-k)(1-r)\\
&\textstyle\quad\quad\quad\quad-\exp(k^2t/2)\hat{f}(k,t)[2\exp(ik1)-(1+r)]\\
&\textstyle\quad\quad\quad\quad\quad-\exp(k^2t/2)\hat{f}(-k,t)(1-r)\Big\}.
\end{align*}
Since $\Upsilon(k)$ is zero whenever $\cos(k)=1$, before we substitute these relations into our integral solution we deform the contours $\partial D^+$ and $\partial D^-$ as shown in Figure \ref{domains} to avoid the poles of $\Upsilon(k)^{-1}$ along the real line.

\begin{figure}
\centering
\includegraphics[height=45mm,trim={0mm 0mm 0mm 0mm},clip]{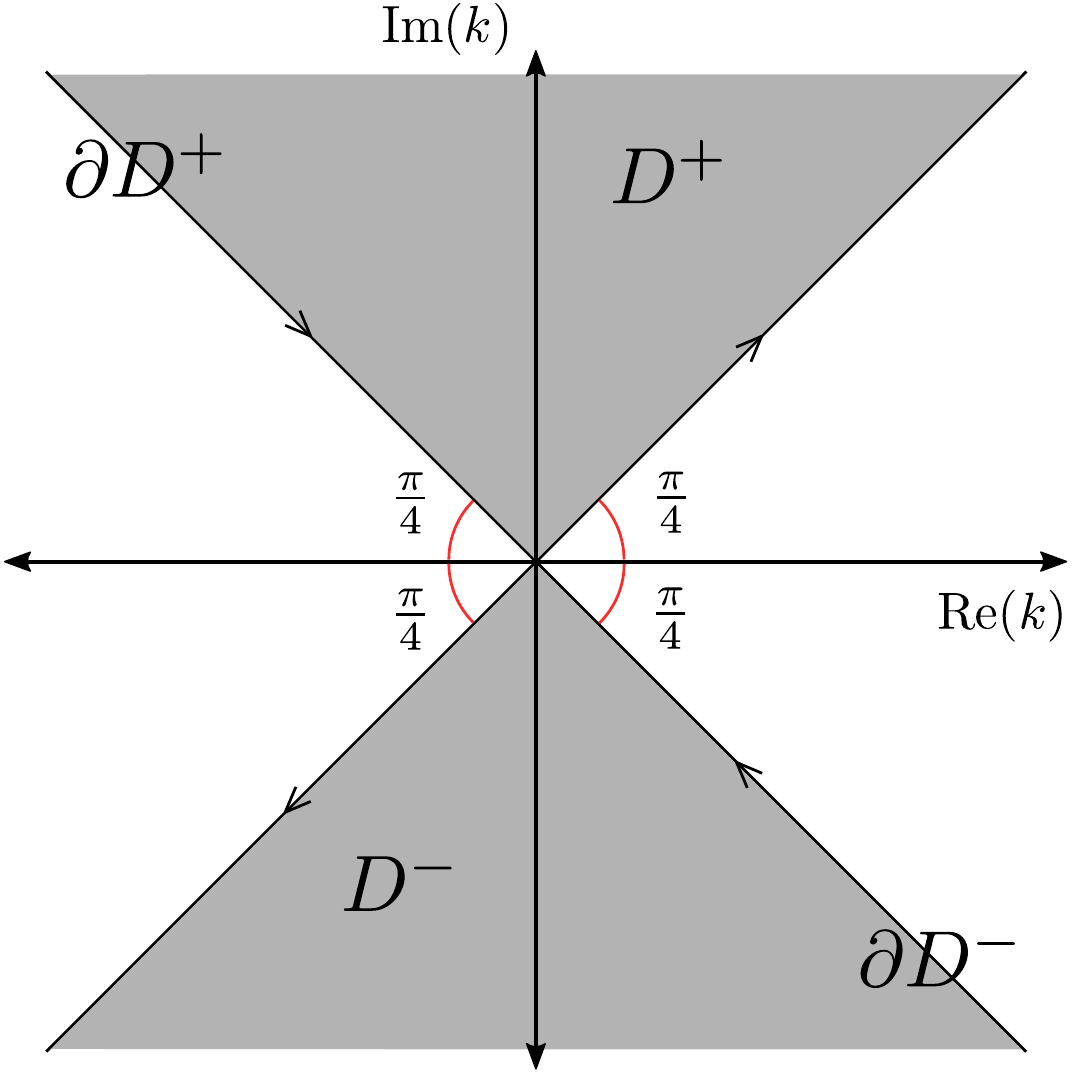}
\hspace{15mm}
\includegraphics[height=45mm,trim={0mm 0mm 0mm 0mm},clip]{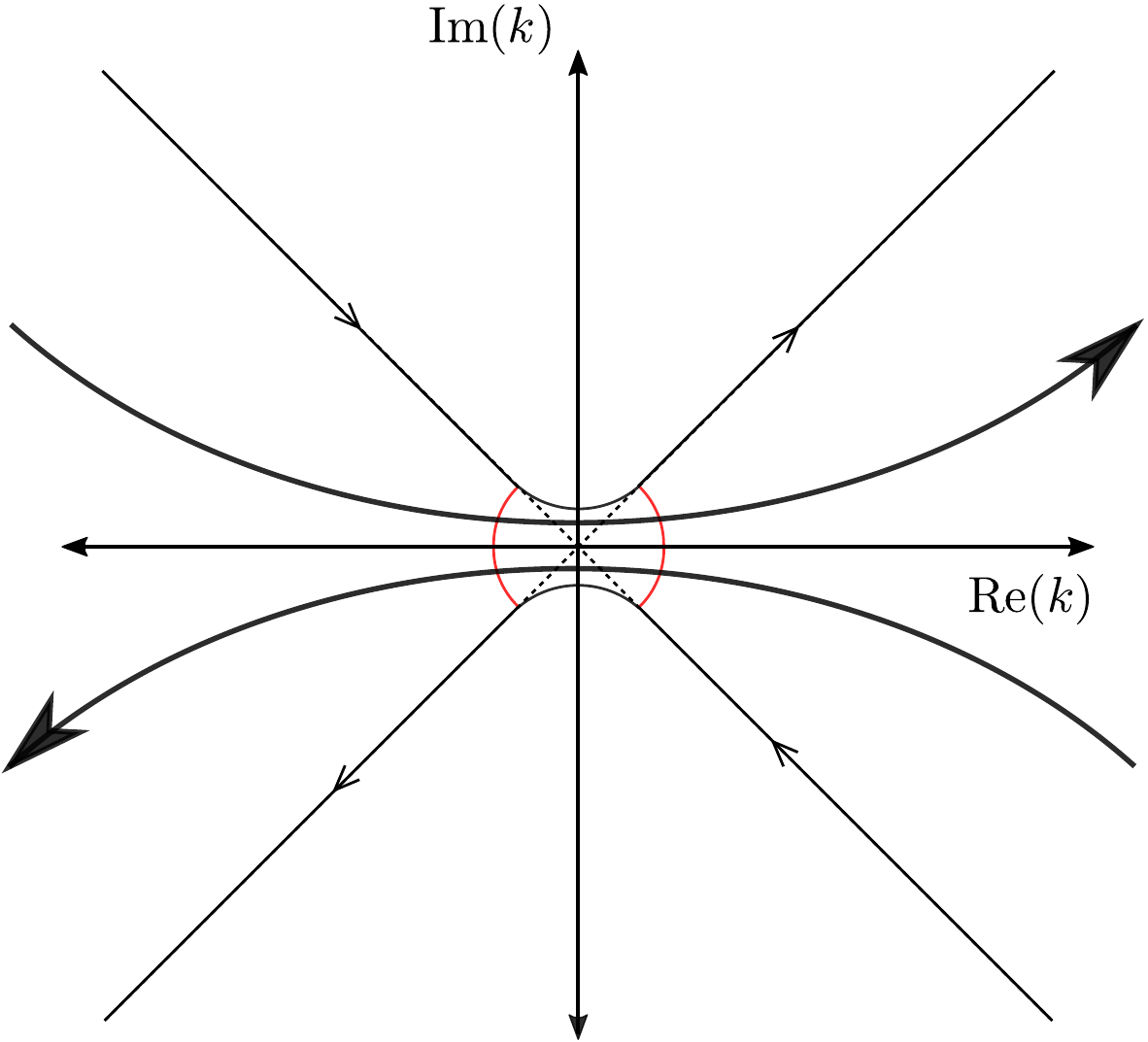}
\caption{Left: The domains $D^{\pm}$ as well as the orientation of the boundaries $\partial D^{\pm}$. Right: The deformed contours to avoid the singularity at $k=0$. The bold arrow shows a path on which both the $x$ and $t$ exponential parts of the integrand are exponentially decaying which can be used for efficient numerical evaluation.}
\label{domains}
\end{figure}

Upon substitution, we are still left with unknown contributions proportional to
\begin{align*}
&\textstyle I_1(x,t):=\int_{\partial D^+}\frac{\exp(ikx)}{\Upsilon(k)}\big\{\hat{f}(k,t)[(1+r)\exp(ik)-2r]\!+\hat{f}(-k,t)(1-r)\exp(-ik)\big\}dk\\
& \textstyle I_2(x,t):=\int_{\partial D^-}\frac{\exp(ik(x-1))}{\Upsilon(k)}\big\{\hat{f}(k,t)[2\exp(ik)-(1+r)]+\hat{f}(-k,t)(1-r)\big\}dk.
\end{align*}
We will argue that the integral $I_1(x,t)$ along $\partial D^+$ vanishes and the argument for $I_2(x,t)$ follows the same reasoning. First observe that as $k\rightarrow\infty$ in $\mathbb{C}^+$, $\Upsilon(k)^{-1}\sim \exp(ik)/(1+r)$. Also, we must have that
\begin{align*} \textstyle
\exp(ik)\hat{f}(k,t)=\int_0^1\exp(ik(1-x))f(x,t)dx
\end{align*}
is bounded in $\mathbb{C}^+$. $\hat{f}(-k,t)$ is also bounded in $\mathbb{C}^+$ and hence the function
$$ \textstyle
\frac{\hat{f}(k,t)[(1+r)\exp(ik)-2r]+\hat{f}(-k,t)(1-r)\exp(-ik)}{\Upsilon(k)}
$$
is bounded in $\mathbb{C}^+$. It follows that we can close the contour in the upper half plane and use Jordan's lemma to see that $I_1(x,t)$ vanishes. We then obtain the integral form of the solution in Theorem \ref{rep_thm}. 


To obtain the series form we can write
$
2\exp(ik)-(1+r)=-\exp(ik)\Upsilon(k)+\exp(ik)[(1+r)\exp(ik)-2r],
$
which implies
\begin{align*}
&\textstyle\int_{\partial D^-}\frac{\exp(ik(x-1)-k^2t/2)}{\Upsilon(k)}\hat{f}_0(k)[2\exp(ik)-(1+r)]dk=\\
&\textstyle\int_{\partial D^-}\exp(ikx-k^2t/2)\hat{f}_0(k)dk
-\!\!\int_{\partial D^-}\!\!\! \frac{\exp(ikx-k^2t/2)}{{\Upsilon(k)}}\hat{f}_0(k)[(1+r)\exp(ik)-2r]dk.
\end{align*}
Taking into account the orientation of $\partial D^{-}$, upon deforming the first of these integrals back to the real line, we see that it cancels the first integral in (\ref{thm_st}). 
Hence we have
\begin{equation}
\label{solution2}\textstyle
2\pi f(x,t)=-\int_{\partial D}\!\!\! \frac{e^{ikx-k^2t/2}}{{\Upsilon(k)}}\big\{\hat{f}_0(k)[(1+r)e^{ik}-2r]+\hat{f}_0(-k)(1-r)e^{-ik}\big\}dk.
\end{equation}
Define the function
\begin{align*}\textstyle
F(x,t;k):=\frac{e^{ikx-k^2t/2}\big\{\hat{f}_0(k)[(1+r)e^{ik}-2r]+\hat{f}_0(-k)(1-r) e^{-ik}\big\}}{2(1+r)}.
\end{align*}
The integrand in (\ref{solution2}) has a double pole at $k_n=2n\pi$ so we deform the contour $\partial D$ to $\partial \tilde{D}$ shown in Figure \ref{residue}. Cauchy's residue theorem then implies that
\begin{equation}
\label{series1}\textstyle
f(x,t)=-\frac{1}{2\pi}\int_{\partial D}\frac{F(x,t;k)}{\cos(k)-1} dk=\sum_{n\in\mathbb{Z}}-2iF'(k_n).
\end{equation}
It is then straightforward to check the equality of (\ref{series1}) and (\ref{series2}).

\begin{figure}
\centering
\includegraphics[height=30mm,trim={0mm 0mm 0mm 0mm},clip]{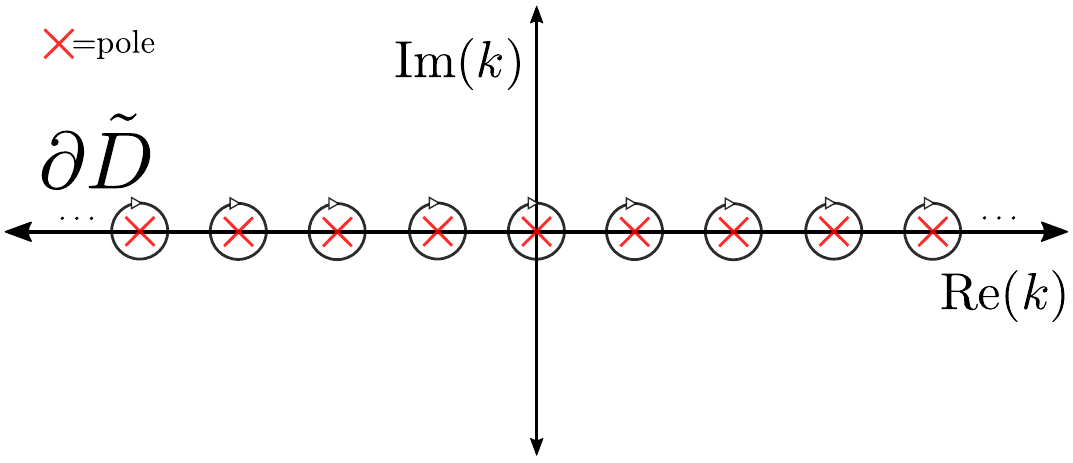}
\caption{Deformation of the contour to circle the poles. The contributions along the real line between these circles cancel.}
\label{residue}
\end{figure}

\subsection{Proof of Theorems \ref{wk_thm} and \ref{rep_thm}}
\label{well-posedness}

\begin{proof}[Proof of Theorems \ref{wk_thm} and \ref{rep_thm}]
For $t>0$, it is clear that the function $f$ given by (\ref{thm_st}) is smooth in $x,t$ and real analytic in $x$, as well as solving the heat equation. This follows from being able to differentiate under the integral sign due to the $\exp(-k^2t/2)$ factor and the fact that extending $x$ to a complex argument yields an analytic function. Note also that the argument in Section \ref{sol_deriv} does rigorously show equivalence between the series and integral forms of $f$. It is easy to check via the series (\ref{series2}) that the function $f$ satisfies the required boundary conditions and hence (\ref{weak_sol}) also holds by simple integration by parts. Regarding the convergence properties as $t\downarrow0$ when extra regularity of the initial condition is assumed, Proposition \ref{cty_bdy} deals with the case of continuous $f_0$, whilst Proposition \ref{cty_bdy2} deals with $f_0\in L^p([0,1])$ for $1\leq p<\infty$.

Hence there are two things left to prove; the fact that $\mu_t:=f(\cdot,t)dx$ lies in $C^w(0,T;M([0,1]))$ as well as uniqueness in $C^w(0,T;M([0,1]))$ (and $C(0,T;L^p([0,1]))$ for $1\leq p<\infty$).

To prove that $\mu_t\in C^w(0,T;M([0,1]))$, let $g\in C([0,1])$ and consider the integral kernel defined by (\ref{kernel0}). By Fubini's theorem we have
$$ \textstyle
\int_0^1 f(x,t)g(x)dx=\int_0^1\int_0^1 K(r;x,y,t)g(x)dxdf_0(y).
$$
By Proposition \ref{cty_bdy} the integral
$$ \textstyle
\int_0^1 K(r;x,y,t)g(x)dx
$$
converges for all $x$ and is uniformly bounded as $t\downarrow0$. We will use the explicit calculation of the endpoints limits at $x=0,1$. By the dominated convergence theorem, we have
\begin{align*} \textstyle
\lim_{t\downarrow 0}&\textstyle\int_0^1 f(x,t)g(x)dx=\big(\frac{r}{1+r}g(0)+\frac{1}{1+r}g(1)\big)f_0(\{0\})\\
&+\textstyle\big(\frac{r}{1+r}g(0)+\frac{1}{1+r}g(1)\big)f_0(\{1\})+\int_{x\in(0,1)}g(x)df_0(x)\\
&=\textstyle f_0(g)+\frac{g(1)-g(0)}{1+r}[f_0(\{0\})-rf_0(\{1\})]=f_0(g),
\end{align*}
which proves the required weak continuity.

To prove uniqueness, suppose that there exists $\mu_t,\tau_t\in M([0,1])$ which are both weak solutions with $\mu_0=\tau_0=f_0$. Set $m_t=\mu_t-\tau_t$. We will consider expansions of functions in the generalized eigenfunctions of the adjoint problem. It is straightforward to check that the adjoint problem (with the boundary conditions in (\ref{adjoint_bcs})) is Birkhoff regular and hence the generalized eigenfunctions are complete in $L^2([0,1])$. In fact we can show that any continuous function $g\in C([0,1])$ of bounded variation with $g(0)=g(1)$ can be approximated uniformly by linear combinations of these functions. This follows by either arguing as we did in the proof of Proposition \ref{cty_bdy} (the case of non-matching derivatives holds but is more involved) or follows from Theorem 7.4.4 of \cite{mennicken2003non}. 
Now suppose that $\lambda$ lies in the spectrum of the adjoint $\mathbb{A}^*$ defined by
\begin{equation} \textstyle
\mathbb{A}^*=-\frac{d^2}{dx^2},\quad \mathcal{D}(\mathbb{A}^*)=\{u\in H^2([0,1]):u_x(1)= ru_x(0),u(0)= u(1)\}.
\label{eq:adjoint:A:star}
\end{equation}
In our case, the generalized eigenfunctions associated with $\lambda$ correspond to a basis of $\mathcal{N}((\mathbb{A}-\lambda I)^l)$ where $l=1$ or $2$. 
If $l=2$, and the nullity of $(\mathbb{A}-\lambda I)^2$ is greater than $\mathbb{A}-\lambda I$, we can choose a basis $\{g_1,g_2\}$ such that $(\mathbb{A}-\lambda I)g_2=g_1$. 
For the general case and chains of generalized eigenfunctions, we refer the reader to \cite{mennicken2003non}. Now suppose that $g\in \mathcal{N}(\mathbb{A}-\lambda I)$, then $g$ must be smooth on $[0,1]$. It follows that for $t>0$
$$ \textstyle
2\frac{d}{dt}m_t(g)=-\lambda m_t(g).
$$
Note that $m_0(g)=0$ and hence we must have that $m_t(g)=0$ for all $t\geq 0$. Similarly, suppose that $\{g_1,g_2\}\subset\mathcal{N}((\mathbb{A}-\lambda I)^2)$ with $(\mathbb{A}-\lambda I)g_2=g_1$. Then by the above reasoning we have $m_t(g_1)=0$ for all $t\geq 0$ and hence
$$ \textstyle
2\frac{d}{dt}m_t(g_2)=-\lambda m_t(g_2)-m_t(g_1)=-\lambda m_t(g_2).
$$
Again we see that $m_t(g_2)=0$ for all $t\geq 0$. Though we don't have to consider it in our case, it is clear that the same argument would work for chains of longer lengths. The expansion theorem discussed above together with the dominated convergence theorem shows that if $g\in C([0,1])$ of bounded variation with $g(0)=g(1)=0$, then $m_t(g)=0$ for all $t\geq 0$. This implies that if $U\subset (0,1)$ is open then $m_t(U)=0$ for all $t\geq 0$. In particular, we must have
$$
m_t=a(t)\delta_0+b(t)\delta_1
$$
with $a,b$ continuous. In fact, for any $f\in\mathcal{F}(r)$ we have
$$ \textstyle
\frac{d}{dt}[a(t)+b(t)]f(1)=\frac{a(t)}{2}f_{xx}(0)+\frac{b(t)}{2}f_{xx}(1),
$$
from which we easily see that $a=b=0$ and hence uniqueness follows. This also shows uniqueness in the space $C(0,\mathbb{A};L^p([0,1]))$, where no argument at the endpoints is needed.
\end{proof}

\subsection{Proof of Theorem \ref{thm:boundary consistency}}
\label{consist_append}

The proof requires that we study the solution of the PDE as $t\downarrow0$. We break down the proof into a number of smaller results, which allows us to use them elsewhere. Recall the definition in (\ref{kernel0}). We shall also need the function
\begin{equation}
\label{f_def}\textstyle
K_1(x,t):=\sum_{n\in\mathbb{Z}}{\exp(ik_nx-k_n^2t/2)},
\end{equation}
defined for $t>0$. Using the Poisson summation formula, we can write $K_1$ as
$$\textstyle
K_1(x,t)=\frac{1}{\sqrt{2\pi t}}\sum_{n\in\mathbb{Z}}\exp\Big(-\frac{(x-n)^2}{2t}\Big),
$$
a periodic summation of the heat kernel. The following lemma is well-known and hence stated without proof.

\begin{lemma}
\label{heat_kernel_classic}
Let $w\in C([0,1])$ then
\begin{equation}
\label{classic2}\textstyle
\int_0^1 K_1(x-y,t) w(y)dy
\end{equation}
is bounded by $\|w\|_{\infty}$ and converges pointwise to $w(x)$ for any $x\in (0,1)$ and to $(w(0)+w(1))/2$ for $x=0,1$ as $t\downarrow0$. If $w(0)=w(1)$ then (\ref{classic2}) converges to $w(x)$ uniformly over the interval $[0,1]$.
\end{lemma}

We will also need the following.

\begin{lemma}
\label{no_t}
Let $f_0\in L^1([0,1])$, then
$$ \textstyle
t\sum_{n\in\mathbb{N}}\left|\exp(-k_n^2t/2)k_n\hat{f}_0(k_n)\right|\rightarrow 0\quad\text{ as }\quad t\downarrow0.
$$
\end{lemma}

\begin{proof}
By the Riemann--Lebesgue lemma, we have that $\lim_{n\rightarrow\infty}\hat{f}_0(k_n)=0$. So given $\epsilon>0$, let $N$ be large such that if $n\geq N$ then $\left|\hat{f}_0(k_n)\right|\leq \epsilon$. Then
$$ \textstyle
t\sum_{n>N}\left|\exp(-k_n^2t/2)k_n\hat{f}_0(k_n)\right|\leq  \frac{t\epsilon}{2\pi}\sum_{n>N}\exp(-2n^2\pi^2t)4n\pi^2.
$$
Let $h=2\pi$. 
The sum is an approximation of the integral $\int_{h(N+1)}^{\infty}\exp(-y^2t/2)ydy$ and we have
$$
t\sum_{n>N}\left|\exp(-k_n^2t/2)k_n\hat{f}_0(k_n)\right|\leq  \frac{\epsilon}{2\pi}\int_0^\infty \exp(-y^2t/2){t(y+2h)}dy< \tilde C\epsilon,
$$
for some constant $\tilde C$. It follows that
$$
\limsup_{t\downarrow0}t\sum_{n\in\mathbb{N}}\left|\exp(-k_n^2t/2)k_n\hat{f}_0(k_n)\right|\leq \tilde C\epsilon.
$$
Since $\epsilon>0$ was arbitrary, the lemma follows.
\end{proof}

The following Proposition then describes the limit properties of our constructed solution as $t\downarrow 0$ in the case of continuous initial data.

\begin{proposition}
\label{cty_bdy}
Let $f_0\in C([0,1])$ and $K$ be given by (\ref{kernel0}). For $t\in(0,1]$ define
$$ \textstyle
f(x,t):=\int_0^1K(r;x,y,t)f_0(y)dy,\quad q(x,t):=\int_0^1K(r;y,x,t)f_0(y)dy,
$$
(note the interchange of $x,y$ as arguments of $K$ for the definition of $q$). Then there exists a constant $C$ (dependent on $r$) such that
\begin{equation}
\label{bdd_sol}
\sup_{x\in[0,1],t\in(0,1]} \max\{\left|f(x,t)\right|,\left|q(x,t)\right|\}\leq C\|f_0\|_{\infty}.
\end{equation}
Furthermore,
\begin{align*}
\lim_{t\downarrow0}f(x,t)&=\begin{cases}
f_0(x),\quad x\in(0,1)\\
\frac{r}{1+r}[f_0(0)+f_0(1)],\quad x=0\\
\frac{1}{1+r}[f_0(0)+f_0(1)],\quad x=1
\end{cases}\\
\lim_{t\downarrow0}q(x,t)&=\begin{cases}
f_0(x),\quad x\in(0,1)\\
\frac{1}{1+r}[rf_0(0)+f_0(1)],\quad x=0,1
\end{cases}.
\end{align*}
Finally, in the case that $f_0(0)=r f_0(1)$, $f(x,t)$ converges to $f_0(x)$ uniformly over $x\in[0,1]$ as $t\downarrow0$.
\end{proposition}

\begin{proof}
We can write

\begin{equation}
\label{kernel1}
\begin{split}\textstyle
K(r;x,y,t)&= \textstyle K_1(x-y,t)\big[1+(x-y)\frac{1-r}{1+r}\big]+K_1(x+y,t)(x+y-1)\frac{1-r}{1+r}\\
&\quad\quad\quad\quad+\frac{t(1-r)}{1+r}\big[K_1'(x+y,t)+K_1'(x-y,t)\big].
\end{split}
\end{equation}
Here $'$ means the derivative with respect to the spatial variable.

To study the limit as $t\downarrow0$, we note that we can ignore the terms with a factor of $t$ using Lemma~\ref{no_t}. By a change of variable we have
$$ \textstyle
\int_0^1 K_1(x+y,t)(x+y-1)f_0(y)dy=\int_0^1 K_1(x-y,t)(x-y)f_0(1-y)dy.
$$
The bound (\ref{bdd_sol}) now follows from Lemma \ref{heat_kernel_classic}, as do the pointwise limits from a straightforward somewhat tedious calculation. 

Now suppose that $f_0(0)=r f_0(1)$ and split the initial data as follows:
\begin{equation}
\label{split1}
f_0(x)=f_0(0)+x(1-r)f_0(1)+p_0(x).
\end{equation}
Then $p_0\in C([0,1])$ with the crucial property that $p_0(0)=p_0(1)=0$. Arguing as above and using Lemma \ref{heat_kernel_classic}, we see that the following limit holds uniformly
$$ \textstyle
\lim_{t\downarrow 0}\int_0^1 K(r;x,y,t)p_0(y)dy=p_0(x).
$$
So it only remains to show that
\begin{equation}
\label{want_this} \textstyle
\int_0^1 K(r;x,y,t)[f_0(0)+y(1-r)f_0(1)]dy=f_0(0)+x(1-r)f_0(1).
\end{equation}
Let $l(x)=f_0(0)+x(1-r)f_0(1)$ and set $a=f_0(1)$. An explicit calculation yields
$$ \textstyle
\hat{l}(k)=\frac{i}{k}(\exp(-ik)-r)a+\frac{1}{k^2}(\exp(-ik)-1)a(1-r).
$$
We then have
\begin{align*}\textstyle
\hat{l}_0(k)[(1+r)\exp(ik)-2r]+\hat{l}_0(-k)(1-r)\exp(-ik)=&\textstyle-a\big[\frac{ri}{k}+\frac{1-r}{k^2}\big]\Upsilon(k).
\end{align*}
We can then apply the residue theorem to the representation (\ref{solution2}) to obtain (\ref{want_this}).
\end{proof}

In the case where the true density is not continuous but belongs to
 $L^p([0,1])$ for $p\geq 1$, we have the following.

\begin{proposition}
\label{cty_bdy2}
Let $1\leq p<\infty$, $f_0\in L^p([0,1])$ and $K$ be given by (\ref{kernel0}). For $t\in(0,1]$ define
\begin{align*} 
f(x,t):=&\textstyle\int_0^1K(r;x,y,t)f_0(y)dy.
\end{align*}
Then $f(\cdot,t)$ converges to $f_0$ in $L^p([0,1])$ as ${t\downarrow0}$.
\end{proposition}

\begin{proof}
Note that the case $r=1$ is well-known. The fact that $f_0 \in L^1([0,1])$ by H{\"o}lder's inequality together with Lemma \ref{no_t} show that we can ignore the parts multiplied by $t$ in the kernel representations \eqref{kernel1}. The fact that $yf_0(y)\in L^p([0,1])$ implies the convergence by simply summing the parts in \eqref{kernel1} and using the $r=1$ case with a change of variable for the $K_1(x+y,t)$ part.
\end{proof}

\begin{proof}[Proof of Theorem \ref{thm:boundary consistency}]
We have that
$$
\mathbb{E}_{f_X}(f(x,t))=\frac{1}{n}\sum_{k=1}^n\int_0^1 K(r;x,y,t)dy=\int_0^1 K(r;x,y,t)dy.
$$
The first part of the theorem therefore follows from Proposition \ref{cty_bdy}. For the second part, assume that $f_X\in C^1([0,1])$ and $x_t=x+\mathcal{O}(\sqrt{t})$. The relation (\ref{missing_piece}) implies the result since we have that
$$
\int_0^t\left|K_1(x,s)\right|ds\leq C\sqrt{t}
$$
for some $C$ independent of $x$ and $\left|f_X(x)-f_X(x_t)\right|=\mathcal{O}(\sqrt{t})$.
\end{proof}

\subsection{Proof of Proposition \ref{max_prin}}
\label{ndprf}

\begin{proof}[Proof of Proposition \ref{max_prin}]
We first show that in this case the solution is continuous on $[0,1]\times[0,T)$ for any $T\in(0,\infty]$. The case of continuity at points $t>0$ has already been discussed so suppose that $(x_n,t_n)\rightarrow(x,0)$ then
$$
\left|f(x_n,t_n)-f_0(x)\right|\leq \left|f_0(x_n)-f_0(x)\right|+\left|f(x_n,t_n)-f_0(x_n)\right|.
$$
The first term converges to zero by continuity of $f_0$ whilst the second term converges to zero by the proven uniform convergence as $t\downarrow0$. Using the limit given by Proposition \ref{large_time}, we will take $T=\infty$ without loss of generality.

Since the solution is regular in the interior and continuous on the closure, this immediately means that we can apply the maximum principle to deduce that
$$
\sup_{(x,t)\in\overline{\Omega}}f(x,t)=\sup_{(x,t)\in\partial{}\Omega}f(x,t),
$$
where $\Omega=(0,1)\times (0,T)$. A similar result holds for the infinum. Evaluating (\ref{series1}) at $x=0$ leads to
$$ \textstyle
f(0,t)=\frac{2r}{1+r}\sum_{n\in\mathbb{Z}}{\exp(-k_n^2 t/2)}\hat{f}_0(k_n)=\frac{r\sqrt{2}}{(1+r)\sqrt{\pi t}}\int_{-\infty}^\infty \exp(-x^2/(2t))f_0(y)dy,
$$
where we have used the function $K_1$ defined by (\ref{f_def}) and extended $f_0$ periodically (values at the endpoints contributed nothing). Hence 
$$ \textstyle
\frac{2ar}{1+r}\leq f(0,t)\leq\frac{2br}{1+r}.
$$
Similar calculations yield
$$ \textstyle
\frac{2a}{1+r}\leq f(1,t)\leq\frac{2b}{1+r}.
$$
The fact $\max\{2r/(1+r),2/(1+r)\}\geq 1$ (recall $r\geq 0$) finishes the proof.
\end{proof}

\subsection{Proof of Theorem \ref{cons_thm}}
\label{ndprf2}
\begin{proof}[Proof of Theorem \ref{cons_thm}]
We have that
$$
\int_0^1f(x,t)dx=\int_0^1\int_0^1K(r;x,y,t)df_0(y)dx=\int_0^1\int_0^1K(r;x,y,t)dxdf_0(y)
$$
by Fubini's theorem. Using the series representation (\ref{kernel1}) and integrating term by term (justified due to the exponential decaying factors) we have
$$
\int_0^1K(r;x,y,t)dx=1+\frac{1-r}{1+r}\sum_{n\in\mathbb{Z}}\exp(-k_n^2t/2)\int_0^1\big[x\exp(ik_n(x-y))+(x-1)\exp(ik_n(x+y))\big]dx.
$$
All other terms vanish since the integral of $\exp(ik_nx)$ is $0$ unless $n=0$. We can change variables for the second term in the integrand to see that the above is equal to
\begin{align*}
&1+\frac{1-r}{1+r}\sum_{n\in\mathbb{Z}}\exp(-k_n^2t/2)\int_0^1\big[x\exp(ik_n(x-y))-x\exp(-ik_n(x-y))\big]dx\\
=&1+\frac{1-r}{1+r}\sum_{n\in\mathbb{Z}}\exp(-k_n^2t/2)\int_0^12ix\sin(k_n(x-y))dx=1,
\end{align*}
where we have used the fact that $\sin(k_n(x-y))$ is odd in $k_n$ and $\exp(-k_n^2t/2)$ is even in the last equality. Since $f_0$ is a probability measure, it follows that $\int_0^1f(x,t)dx=1$, i.e. part (1) holds.

We next show that the integral kernel $K(r;x,y,t)$ is non-negative for $r\geq0,t>0$ and $x,y\in[0,1]$. Suppose this were false for some $(x_0,y_0)\in[0,1]^2$. The Poisson summation formula gives
$$
K(r;0,y,t)=\frac{2r}{1+r}K_1(y,t)>0,\quad K(r;1,y,t)=\frac{2}{1+r}K_1(y,t)>0,
$$
and hence $(x_0,y_0)\in(0,1)^2$. Choose $f_{0}=u_n$ that integrates to $1$ where $u_n(y)\geq 0$ and $u_n(y)=0$ unless $\|y-y_0\|\leq 1/n$. Then for large $n$, $u_n$ satisfies the required boundary conditions (vanishes in a neighborhood of the endpoints) and we must have that
$$
f_n(x_0,t):=\int_0^1K(r;x_0,y,t)u_n(y)dy\geq 0,
$$
by Proposition \ref{max_prin}. But it clearly holds by continuity of the integral kernel that $\lim_{n\rightarrow\infty}f_n(x_0,t)=K(r;x_0,y_0,t)<0$, a contradiction. This proves part (2) of the theorem.
\end{proof}

\section{Proof of Theorem \ref{AMISE_1}}
\label{AMISE_proof}

\begin{proof}
We begin with the proof of 1. Recall that
$$ \textstyle
\mathrm{Var}_{f_X}[f(x,t)]=\frac{\mathbb{E}_{f_X}[K(x,Y,t)^2]}{n}-\frac{\mathbb{E}_{f_X}[K(x,Y,t)]^2}{n},
$$
where $K$ is the kernel given by (\ref{kernel1}). The second of these terms is bounded by a constant multiple of $1/n$ so we consider the first. Recall the decomposition (\ref{kernel1}):
\begin{equation*}
\begin{split} \textstyle
K(r;x,y,t)&= \textstyle K_1(x-y,t)\big[1+(x-y)\frac{1-r}{1+r}\big]+K_1(x+y,t)(x+y-1)\frac{1-r}{1+r}\\
&\quad\quad\quad\quad+\frac{t(1-r)}{1+r}\big[K_1'(x+y,t)+K_1'(x-y,t)\big],
\end{split}
\end{equation*}
where $K_1$ is the standard periodic heat kernel. For $x,y\in[0,1]$ we have that
\begin{align*}
K_1(x-y,t)&\textstyle\sim_{t\downarrow0}\frac{1}{\sqrt{2\pi t}}\big[\exp\big(-\frac{(x-y)^2}{2t}\big)+\exp\big(-\frac{(x-y-1)^2}{2t}\big)+\exp\big(-\frac{(x-y+1)^2}{2t}\big)\big]\\
K_1(x+y,t)&\textstyle\sim_{t\downarrow0}\frac{1}{\sqrt{2\pi t}}\big[\exp\big(-\frac{(x+y)^2}{2t}\big)+\exp\big(-\frac{(x+y-1)^2}{2t}\big)+\exp\big(-\frac{(x+y-2)^2}{2t}\big)\big],
\end{align*}
with the rest of the expansion exponentially small as $t\downarrow0$ and the asymptotics valid upon taking derivatives. Using this, it is straightforward to show that we can write
$$ \textstyle
K(r;x,y,t)=\frac{1}{\sqrt{2\pi t}}G(r;x,y,t),
$$
where $G$ is bounded. From the above asymptotic expansions, we can write
\begin{align*}
G(r;x,y,t)&=\textstyle\exp\Big(-\frac{(x-y)^2}{2t}\Big)+h_1(x,y,t)\exp\Big(-\frac{(x-y-1)^2}{2t}\Big)\\
&\quad \textstyle +h_2(x,y,t)\exp\Big(-\frac{(x-y+1)^2}{2t}\Big)+h_3(x,y,t)\exp\Big(-\frac{(x+y)^2}{2t}\Big)\\
&\quad\quad \textstyle+h_4(x,y,t)\exp\Big(-\frac{(x+y-2)^2}{2t}\Big)+E(x,y,t),
\end{align*}
where the $h_i$ are bounded and the error term $E(x,y,t)$ is exponentially small as $t\downarrow0$ uniformly in $x,y$. Furthermore, we have 
$$ \textstyle
\lim_{t\downarrow 0}\int_0^1\int_0^1f_X(y)K(r;x,y,t)h_1(x,y,t)\exp\Big(-\frac{(x-y-1)^2}{2t}\Big)dydx=0
$$
by the dominated convergence theorem (by considering the inner integral as a function of $x$). Similar results hold for the other $h_i$ multiplied by their relative Gaussian functions. Similarly, we have
$$ \textstyle
\lim_{t\downarrow 0}\frac{1}{\sqrt{t}}\int_0^1\int_0^1\exp\Big(-\frac{(x-y)^2}{2t}\Big)f_X(y)h_1(x,y,t)\exp\Big(-\frac{(x-y-1)^2}{2t}\Big)dydx=0
$$
and likewise for the other $h_i$ multiplied by their relative Gaussian functions. The integral
$$ \textstyle
\frac{1}{\sqrt{\pi t}}\int_0^1f_X(y)\exp\Big(-\frac{(x-y)^2}{t}\Big)dy
$$
is bounded and converges pointwise for almost all $x\in[0,1]$ to $f_X(x)$. It follows that
\begin{equation}\textstyle
\label{Var_asym2}
\int_0^1\frac{\mathbb{E}_{f_X}[K(x,Y,t)^2]}{n}dx=\frac{1}{2n\pi t}\int_0^1\int_0^1f_X(y)\exp\Big(-\frac{(x-y)^2}{t}\Big)dydx+\underline{\text{o}}(\frac{1}{n\sqrt{t}}).
\end{equation}
The rate (\ref{Var_asym}) now follows.

We now prove 2 and 3. Define the function $p_0(x)$ via $f_X(x)=f_X(0)+x(1-r)f_X(1)+p_0(x)$, then the proof of Proposition \ref{cty_bdy} showed that
\begin{equation}
\label{expectation_rep}
\mathbb{E}_{f_X}[f(x,t)]-f_X(x)=\int_0^1K(r;x,y,t)[p_0(y)-p_0(x)]dy.
\end{equation}
Define the function
\begin{equation}
\label{aux_function}
w(x,y)=p_0(y)-p_0(x)+(x-y)\frac{1-r}{1+r}(p_0(y)+p_0(1-y)),
\end{equation}
then (\ref{expectation_rep}) and (\ref{kernel1}) imply that $\mathbb{E}_{f_X}[f(x,t)]-f_X(x)$ is equal to
\begin{equation}
\label{expectation_rep2}
\textstyle
\int_0^1K_1(x-y,t)w(x,y)dy+\frac{t(1-r)}{1+r}\int_0^1K_1(x-y,t)[p_0'(y)-p_0'(1-y)]dy,
\end{equation}
where we have integrated by parts for the last term and used $p_0(0)=p_0(1)=0$. Define the function
\begin{equation} \textstyle
F(x,t)=\int_0^1K_1(x-y,t)w(x,y)dy,
\end{equation}
then taking the partial derivative with respect to time, integrating by parts and using $w(x,0)=w(x,1)=0$ we have
\begin{equation}\textstyle
\frac{\partial F}{\partial t}(x,t)=K_1(x,t)[p_0'(1)-p_0'(0)]\frac{rx-x-r}{1+r}+\frac{1}{2}\int_0^1K_1(x-y,t)\frac{\partial^2w}{\partial y^2}(x,y)dy.
\end{equation}
First we assume that $p_0'(0)\neq p_0'(1)$. In this case the above shows that
\begin{equation}\textstyle
\mathbb{E}_{f_X}[f(x,t)]-f_X(x)=\frac{rx-x-r}{1+r}[p_0'(1)-p_0'(0)]\int_0^tK_1(x,s)ds +\mathcal{O}(t),
\label{missing_piece}
\end{equation}
where the $\mathcal{O}(t)$ is uniform in $x$. Using the above asymptotics for $K_1(x,t)$ and the dominated convergence theorem, it follows that
\begin{equation}
\begin{split}
&\textstyle \int_0^1\big\{\mathbb{E}_{f_X}[f(x,t)]-f_X(x)\big\}^2dx \\
&\sim_{t\downarrow0} \textstyle\frac{[p_0'(1)-p_0'(0)]^2}{2\pi(1+r)^2}\int_0^1(rx-x-r)^2\Big[\int_0^t\frac{\exp\big(-\frac{x^2}{2s}\big)}{\sqrt{s}}+\frac{\exp\big(-\frac{(x-1)^2}{2s}\big)}{\sqrt{s}}ds\Big]^2dx.
\end{split}
\end{equation}
Let $\tau(x)=\exp(-x^2)-\sqrt{\pi}\left|x\right|\mathrm{erfc}(\left|x\right|)$, then we can perform the integral in the square brackets in terms of $\tau$ to yield
\begin{align}
&\textstyle\int_0^1\big\{\mathbb{E}_{f_X}[f(x,t)]-f_X(x)\big\}^2dx\\
&\textstyle\sim_{t\downarrow0}t\Big\{\frac{2[p_0'(1)-p_0'(0)]^2}{\pi(1+r)^2}\int_0^1(rx-x-r)^2\big[\tau(\frac{x}{\sqrt{2t}})+\tau(\frac{x-1}{\sqrt{2t}})\big]^2dx\Big\}\\
&\textstyle\sim_{t\downarrow0}t^{3/2}\Big\{\frac{2[p_0'(1)-p_0'(0)]^2}{\pi(1+r)^2}(r^2+1)\sqrt{2}\int_0^\infty\tau(y)^2dy\Big\}.
\end{align}
To finish the proof in this case, we have that
$$ \textstyle
\int_0^\infty\tau(y)^2dy=\frac{1}{3}(\sqrt{2}-1)\sqrt{\pi}.
$$
Next suppose that $p_0'(0)= p_0'(1)$. 
In this case we have
\begin{align}\textstyle
\frac{\partial F}{\partial t}(x,t) & = \textstyle\frac{1}{2} \int_0^1 K_1(x-y,t) \frac{\partial^2w}{\partial y^2}(x,y)dy \\ 
& = \textstyle\frac{1}{2}\frac{\partial^2 w}{\partial y^2}(x,x)+U(x,t),
\end{align}
for some bounded function $U(x,t)$ which converges to $0$ as $t\downarrow0$ for almost all $x\in[0,1]$. It follows that
\begin{equation}\textstyle
F(x,t)=t\frac{1}{2}\frac{\partial^2w}{\partial y^2}(x,x)+t\tilde{F}(x,t),
\end{equation}
for some bounded function $\tilde{F}(x,t)$ which converges to $0$ as $t\downarrow0$ for almost all $x\in[0,1]$. Similarly, we have
\begin{equation}\textstyle
\frac{t(1-r)}{1+r}\int_0^1K_1(x-y,t)[p_0'(y)-p_0'(1-y)]dy=\frac{t(1-r)}{1+r}[p_0'(x)-p_0'(1-x)]+tV(x,t),
\end{equation}
for some bounded function $V(x,t)$ which converges to $0$ as $t\downarrow0$ for almost all $x\in[0,1]$. It follows from (\ref{expectation_rep2}) that
\begin{equation}\textstyle
\label{expectation_rep3}
\mathbb{E}_{f_X}[f(x,t)]-f_X(x)=t\Big\{\frac{1}{2}\frac{\partial^2w}{\partial y^2}(x,x)+\frac{(1-r)}{1+r}[p_0'(x)-p_0'(1-x)]\Big\}+tW(x,t),
\end{equation}
for some bounded function $W(x,t)$ which converges to $0$ as $t\downarrow0$ for almost all $x\in[0,1]$. But we have
$$ \textstyle
\frac{\partial^2w}{\partial y^2}(x,x)=f_X^{''}(x)-2\frac{1-r}{1+r}[p_0'(x)-p_0'(1-x)].
$$
The dominated convergence theorem then implies that
\begin{equation}\textstyle
\label{expectation_rep4}
\int_0^1\{\mathbb{E}_{f_X}[f(x,t)]-f_X(x)\}^2dx\sim_{t\downarrow0}t^2\int_0^1\frac{1}{4}[f_X^{''}(x)]^2dx.
\end{equation}
\end{proof}

\section{Four Corners Matrix and Proof of Theorem~\ref{thm:discrete:model:converges}}
\label{4 corners matrix}

The `Four Corners Matrix' \eqref{eq:four:corners:matrix}, is a \textit{non-symmetric} example of a `tridiagonal Toeplitz matrix with four perturbed corners' \cite{yueh2008explicit,StrMac14}. Although we do not pursue it further, one can also show (by extending the techniques of \cite{StrMac14}) that all functions of \eqref{eq:four:corners:matrix} are the sum of (i) a Toeplitz part, which can be thought of as the solution without boundary conditions; and (ii) a Hankel part, which is precisely the correction due to the boundary conditions. Exact and explicit formulas for the eigenvalues and eigenvectors are available and we will use these to prove Theorem~\ref{thm:discrete:model:converges}
 
There is a unique zero eigenvalue, corresponding to the stationary density as $t \rightarrow \infty$. The stationary density is an affine function in the continuous PDE setting. In the discrete setting the components of the stationary eigenvector $\bm{v}$ are equally-spaced, in the sense that $\forall i,j,  \; \; v_{i} - v_{i+1} =  v_{j} - v_{j+1}  = \textrm{constant}$. All non-zero eigenvalues of $\m A$ are positive and we are in the setting of \cite[Theorem 3.2 (i)]{yueh2008explicit}. 
In the case that $r\neq1 $, we can group the spectral data into two classes with eigenvalues:
\[
\lambda_k = 2- 2 \cos \theta_k , \qquad k = 1,\ldots, m,
\]
where
\[
\theta_k =
\begin{cases}
k \frac{2 \pi }{m}  &  \mbox{     if } 1 \le k \le \floor*{\frac{m-1}{2}} \\
 (k - \floor*{\frac{m-1}{2}}  - 1) \frac{2 \pi}{m+1}  & \mbox{     if } \floor*{\frac{m-1}{2}} +1\le k \le m.
\end{cases}
\]
The zero eigenvalue, when $k=\floor*{\frac{m-1}{2}} +1$, has already been discussed.
Other eigenvalues correspond to eigenvectors with components (listed via subscripts)
\begin{equation}
\begin{cases}
v^k_j= r\sin((j-1)\theta_k)-\sin(j\theta_k)  &  \mbox{     if } 1 \le k \le \floor*{\frac{m-1}{2}} \\
w^{k-\floor{\frac{m-1}{2}} -1}_j=\sin(j\theta_k)  & \mbox{     if } \floor*{\frac{m-1}{2}} +2\le k \le m.
\end{cases}
\label{eq:eigen:vector:formulae}
\end{equation}

Some properties of the discrete model and its links to the continuous model are:
\begin{itemize}
\item {All eigenvalues of the Four Corners Matrix are purely real.} 
Also, the eigenvalues of the operator in the continuous model are likewise purely real. This is perhaps surprising since the matrix is not symmetric, and the operator is not self-adjoint.
\item The Four Corners Matrix $\m A$ is diagonalizable. In contrast, the operator for the continuous PDE is not diagonalizable, and instead, the analogy of the Jordan Normal Form from linear algebra applies to the operator. Despite this, the following still hold:
\begin{enumerate}
	\item The eigenvalues of the discrete model matrix $\m A$ converge to that of the continuous model (including algebraic multiplicities). This holds, for example, in the Attouch--Wets topology - the convergence is locally uniform.
	\item The eigenvectors converge to the generalized eigenfunctions of the continuous operator. Letting $j=\floor{(m+1)x}$ we have ($k\neq0$)
	\begin{align*}
	&\lim_{m\rightarrow\infty} w^k_j= \sin(2\pi kx)\\
	&\lim_{m\rightarrow\infty}\frac{m}{4\pi^2k^2}\big[(r-1)w^k_j-v^k_j\big]=\phi_k(x).
	\end{align*}
\end{enumerate}
\end{itemize}

We prove Theorem~\ref{thm:discrete:model:converges} by invoking the celebrated Lax Equivalence Theorem, which states that `stability and consistency implies convergence' \cite{richtmyer1994difference}. We will take consistency for granted. Typically when proofs in the literature use the Lax Equivalence Theorem, it is also taken for granted that the PDE is well-posed. Fortunately, we have already established that the PDE is indeed well-posed in Theorem~\ref{wk_thm}. It remains only to show stability. Even though the matrix $\m A$ has non-negative eigenvalues, this does not immediately imply stability of the backward Euler method since $\m A$ is not normal, i.e. $\m A$ does not commute with $\m A^*$. We establish stability for our problem by showing that bounds for the continuous model in Proposition \ref{max_prin} have corresponding bounds in the discrete model as follows, where we use a subscript $l^p$ to denote the operator $l^p$ norm. In particular, Lemma~\ref{thm:discrete_bound} shows the discrete model is \textit{stable} in the maximum norm. Convergence then follows from the Lax Theorem.

\begin{lemma}[Stability and Well-posedness of Discrete Approximation]
\label{thm:discrete_bound}
Let $m\geq 2$, then the backward Euler method \eqref{backwards_euler} preserves probability vectors and satisfies the bound
\begin{equation*}
\left\|\left(\m I +\m A\right)^{-K}\right\|_{l^\infty}\leq\max\Big\{\frac{2r}{1+r},\frac{2}{1+r}\Big\},\quad \forall K\in\mathbb{N}.
\end{equation*}
 \end{lemma}
As a result of this lemma, we also gain stability in any $p$--norm via interpolation.

\begin{proof}[Proof of Lemma \ref{thm:discrete_bound}]
Since the sums of each column $\m A$ are zero, it follows that
$$
\sum_{j=1}^mu_j^{k+1}=\sum_{j=1}^mu_j^{k}.
$$
Hence to prove the first part, it suffices to show that $\v u^{k+1}$ is non-negative if $\v u^{k}$ is. Suppose this were false, and let $j\in\{1,..,m\}$ be such that $u^{k+1}_j<0$ is the smallest component of $\v u^{k+1}$. We have that
$$
u_j^{k}=(1+\m A_{jj})u_j^{k+1}+\sum_{l\neq j}\m A_{j,l}u_l^{k+1}.
$$
By choice of $j$ and the fact that $\m A_{jj}$ is positive, the off-diagonals of $\m A$ are negative and the sum of the $j$th column of $\m A$ is zero, it follows that
$$
\m A_{jj}u_j^{k+1}+\sum_{l\neq j}\m A_{j,l}u_l^{k+1}\leq 0.
$$
But this then implies that $u_j^{k}\leq u_j^{k+1}$, the required contradiction.

To prove the second part, let $\v u\in\mathbb{R}_{\geq0}^m$ be any initial vector with $\|\v u\|_{\infty}\leq 1$ and let $\mathbbm{1}$ denote the vector with $1$ in all entries. The eigenvector in the kernel of $\m A$ is a linear multiple of $w^0$ defined by
$$ \textstyle
w^0_j=1+\frac{1-r}{1+rm}(j-1).
$$
Define the vector $x$ via $x(1-r)/(1+r)= \mathbbm{1} -2(1+rm)/[(m+1)(r+1)]w^0$. 
This has components
$$ \textstyle
x_j=\frac{m+1-2j}{m+1}.
$$
Extend this vector to have $x_0=0$, then an application of the discrete Fourier transform implies that we can write for $j\neq 0$
$$\textstyle
x_j=\frac{1}{m+1}\sum_{k=1}^mG_m(k)\exp\left(\frac{2\pi ikj}{m+1}\right),
$$
where
$$ \textstyle
G_m(k)=\frac{\exp(4\pi ik/(m+1))-1}{(\exp(2\pi ik/(m+1))-1)^2}=\frac{-i}{2}\frac{\sin(2\pi k/(m+1))}{\sin^2(k\pi/(m+1))}.
$$
Hence we have that
\begin{align*}
x_j&= \textstyle\frac{1}{m+1}\sum_{k=1}^{m-\floor{\frac{m-1}{2}}-1}\left[G_m(k)\exp\left(\frac{2\pi ikj}{m+1}\right)-\overline{G_m(k)}\exp\left(-\frac{2\pi ikj}{m+1}\right)\right]\\
&=\textstyle\frac{1}{m+1}\sum_{k=1}^{m-\floor{\frac{m-1}{2}}-1}\frac{\sin(2\pi k/(m+1))}{\sin^2(k\pi/(m+1))}\sin\left(\frac{2\pi kj}{m+1}\right).
\end{align*}
This implies that we can write $1$ as a linear combination of eigenvectors:
$$ \textstyle
1=\frac{2(1+rm)w^0_j}{(m+1)(1+r)}+\frac{1-r}{(m+1)(1+r)}\sum_{k=1}^{m-\floor{\frac{m-1}{2}}-1}\frac{\sin(2\pi k/(m+1))}{\sin^2(k\pi/(m+1))}w^k_j.
$$
Define $\v Q^K=\left(\m I+\m A\right)^{-K}\mathbbm{1}$, and $\v q^K=\left(\m I+\m A\right)^{-K}\v u$. In particular, using the eigenvalue decomposition we have
\begin{align*}
Q_j^K&=\textstyle\frac{2(1+rm)w^0_j}{(m+1)(1+r)}\\
& +\frac{1-r}{(m+1)(1+r)}\textstyle \sum_{k=1}^{m-\floor{\frac{m-1}{2}}-1}\frac{\sin(2\pi k/(m+1))}{\sin^2(k\pi/(m+1))}\sin\left(\frac{2\pi kj}{m+1}\right)\left(3-2\cos\left(\frac{2\pi k}{m+1}\right)\right)^{-K}.
\end{align*}
Using similar arguments to the first part of the proof, it is easy to prove the \textbf{Discrete Maximum Principle}:
$$
\sup_{K\in\mathbb{N}\cup\{0\}}\left\|\v Q^K\right\|_{\infty}=\max\left\{\sup_{K\in\mathbb{N}\cup\{0\}}Q_1^K,\sup_{K\in\mathbb{N}\cup\{0\}}Q_m^K,1\right\}.
$$
Explicitly, we have that
\begin{align*}
Q_1^K&=\textstyle\frac{2(1+rm)}{(m+1)(1+r)} \\
& +\frac{1-r}{(m+1)(1+r)} \textstyle\sum_{k=1}^{m-\floor{\frac{m-1}{2}}-1}\frac{\sin^2(2\pi k/(m+1))}{\sin^2(k\pi/(m+1))}\left(3-2\cos\left(\frac{2\pi k}{m+1}\right)\right)^{-K}.
\end{align*}
This is monotonic in $K$ with limit $2(1+rm)/[(m+1)(1+r)]$. 
Similarly, we have 
\begin{align*}
Q_m^K&=\textstyle\frac{2(m+r)}{(m+1)(1+r)}\\
& -\frac{1-r}{(m+1)(1+r)} \textstyle\sum_{k=1}^{n-\floor{\frac{m-1}{2}}-1}\frac{\sin^2(2\pi k/(m+1))}{\sin^2(k\pi/(m+1))}\left(3-2\cos\left(\frac{2\pi k}{m+1}\right)\right)^{-K},
\end{align*}
which is monotonic in $K$ with limit $2(m+r)/[(m+1)(1+r)]$. Now, we must have that each entry of $\v Q^K\pm \v q^K$ is non-negative since this is true for $K=0$. It follows that
$$
\left\|\v q^K\right\|_{\infty}\leq \left\|\v Q^K\right\|_{\infty}=\max\left\{\frac{2(1+rm)}{(m+1)(1+r)},\frac{2(m+r)}{(m+1)(1+r)}\right\}.
$$
Since the $l^\infty$ operator norm of a real matrix is independent of whether the underlying field is $\mathbb{R}$ or $\mathbb{C}$, the lemma now follows by taking suprema over $m$.
\end{proof}

\bibliographystyle{plain} 
\bibliography{boundary}
\end{document}